\documentclass{amsart}
\usepackage{amssymb, amsmath}
\usepackage[all]{xy}
\usepackage{todonotes}
\usepackage{mathrsfs}
\usepackage{url}
\usepackage{graphicx}

\newcommand{\Ninfty}{\mathcal N_{\infty}}
\newcommand{\cO}{\mathcal O}
\newcommand{\m}[1]{\underline{#1}}

\let\sma\wedge
\newcommand{\htp}{\simeq}
\renewcommand{\to}{\mathchoice{\longrightarrow}{\rightarrow}{\rightarrow}{\rightarrow}}

\newcommand{\sI}{\mathscr{I}}

\newcommand{\cC}{{\mathcal C}}
\newcommand{\cD}{{\mathcal D}}

\newcommand{\cF}{{\mathcal F}}

\newcommand{\cP}{{\mathcal P}}

\newcommand{\sL}{{\mathscr L}}
\newcommand{\sK}{{\mathscr K}}

\let\catsymbfont\mathcal
\newcommand{\aA}{{\catsymbfont{A}}}
\newcommand{\aB}{{\catsymbfont{B}}}
\newcommand{\aC}{{\catsymbfont{C}}}

\newcommand{\aF}{{\catsymbfont{F}}}

\newcommand{\aL}{{\catsymbfont{L}}}
\newcommand{\aM}{{\catsymbfont{M}}}

\newcommand{\aO}{{\catsymbfont{O}}}

\newcommand{\bL}{\mathbb{L}}

\newcommand{\bP}{{\mathbb{P}}}

\newcommand{\bR}{{\mathbb{R}}}
\newcommand{\bS}{\mathbb{S}}
\newcommand{\bT}{\mathbb{T}}

\def\quickop#1{\expandafter\DeclareMathOperator\csname
#1\endcsname{#1}}
\quickop{id}\quickop{Id}\quickop{colim}\quickop{hocolim}\quickop{op}
\quickop{co}\quickop{Ar}\quickop{sing}\quickop{Hom}\quickop{w}\quickop{Ho}
\quickop{ob}\quickop{diag}\quickop{Stab}\quickop{Cat}\quickop{Mot}\quickop{Mod}
\quickop{map}\quickop{Cone}\quickop{End}\quickop{Idem}\quickop{Perf}\quickop{Ind}
\quickop{Gap}\quickop{Shv}\quickop{Spaces}\quickop{cof}\quickop{ex}\quickop{perf}
\quickop{tri}\quickop{Set}\quickop{Fib}\quickop{Nat}\quickop{haut}\quickop{holim}\quickop{Tor}\quickop{Ext}\quickop{Diff}\quickop{lax}\quickop{Map}\quickop{Comm}\quickop{Top}\quickop{Tr}\quickop{Res}\quickop{THH}\quickop{TR}\quickop{cyc}\quickop{sd}\quickop{fib}\quickop{cf}\quickop{GL}\quickop{Alg}

\newcommand{\mSet}{\underline{\Set}}
\newcommand{\mF}{\underline{\cF}}

\newcommand{\modules}[1]{#1\text{-}\mbox{Mod}}

\numberwithin{equation}{section}

\newtheorem{theorem}[equation]{Theorem}
\newtheorem*{theorem*}{Theorem}
\newtheorem{corollary}[equation]{Corollary}
\newtheorem{lemma}[equation]{Lemma}
\newtheorem{proposition}[equation]{Proposition}
\theoremstyle{definition}
\newtheorem{definition}[equation]{Definition}

\newtheorem{remark}[equation]{Remark}

\newtheorem{warning}[equation]{Warning}

\xyoption{arrow}
\xyoption{matrix}
\xyoption{cmtip}
\SelectTips{cm}{}

\newcommand{\Sp}{\mathcal Sp}
\newcommand{\gspectra}[1][G]{#1\mathcal{S}}
\newcommand{\Sg}[1][G]{\mathbb{S}_{#1}}

\newdir{ >}{{}*!/-5pt/\dir{>}}

\begin{document}

\title[Modules over $\Ninfty$ rings]{$G$-symmetric monoidal categories
of modules over equivariant commutative ring spectra}

\author[A.~J.~Blumberg]{Andrew~J. Blumberg}
\address{University of Texas \\ Austin, TX 78712}
\email{blumberg@math.utexas.edu}
\thanks{A.~J.~Blumberg was supported in part by NSF grant DMS-1151577}

\author[M.~A.~Hill]{Michael~A. Hill}
\address{University of California Los Angeles\\ Los Angeles, CA 90095}
\email{mikehill@math.ucla.edu}
\thanks{M.~A.~Hill was supported in part by NSF grant DMS-0906285, DARPA
grant FA9550-07-1-0555, and the Sloan Foundation}


\begin{abstract}
We describe the multiplicative structures that arise on categories of
equivariant modules over certain equivariant commutative ring spectra.
Building on our previous work on $\Ninfty$ ring spectra, we construct
categories of equivariant operadic modules over $\Ninfty$ rings that
are structured by equivariant linear isometries operads.  These
categories of modules are endowed with equivariant symmetric monoidal
structures, which amounts to the structure of an ``incomplete Mackey
functor in homotopical categories''.  In particular, we construct
internal norms which satisfy the double coset formula.  One
application of the work of this paper is to provide a context in which
to describe the behavior of Bousfield localization of equivariant
commutative rings.  We regard the work of this paper as a first step
towards equivariant derived algebraic geometry.
\end{abstract}

\maketitle \setcounter{tocdepth}{2} \tableofcontents

\section{Introduction}

Stable homotopy theory has been revolutionized over the last twenty
years by the development of symmetric monoidal categories of
spectra~\cite{EKMM, mmss, HSS}.  Commutative monoids in these
categories model $E_\infty$ ring spectra.  Arguably the most important
consequence of this machinery is the ability to have tractable
point-set models for homotopical categories of modules over an
$E_\infty$ ring spectrum $R$.  In the equivariant setting, analogous
symmetric monoidal categories of $G$-spectra have been constructed,
most notably the category of orthogonal $G$-spectra~\cite{MM}.  Once
again commutative monoids model $E_\infty$ ring spectra and so there
are good point-set models for homotopical categories of modules over
such an equivariant $E_\infty$ ring spectrum.

Modules over a commutative ring orthogonal $G$-spectrum $R$ form a
``$G$-symmetric monoidal category''~\cite{HillHopkinsLocalization}.
Roughly speaking, for each $G$-set $T$, we have an internal norm in
the category of $R$-modules; for an orbit $G/H$, the internal norm is
precisely the composite of the $R$-relative norm $_R N_H^G$ and the
forgetful functor from $R$-modules to $\iota_H^* R$-modules.  These
internal norms are compatible with disjoint unions of $G$-sets and
with restrictions to subgroups, and if the set has a trivial action
and cardinality $n$, then we simply recover the smash power functors
$X \mapsto X^n$.

However, in contrast to the non-equivariant setting, there are many
possible notions of $E_\infty$ ring spectra when working over a
nontrivial finite group $G$.  The commutative monoids in orthogonal
$G$-spectra are just one end of the spectrum of possible multiplicative
structures.  In a previous paper, we described this situation in
detail by explaining how such multiplications can be structured by
$\Ninfty$ operads~\cite{BlumbergHill}.  Roughly speaking, just as a
commutative ring is characterized by compatible multiplication maps
$R^{\wedge n} \to R$ as $n$ varies over the natural numbers, a
commutative $G$-ring is characterized by compatible equivariant
multiplication maps $R^{\wedge T} \to R$, where here $T$ is a $G$-set.
The $\Ninfty$ operads structure which such equivariant norms exist for
a given commutative ring, expressed in terms of compatible families of
subgroups of $G \times \Sigma_n$.  Specifically, associated to an
$\Ninfty$ operad $\aO$ there is a coefficient system $\aC(\aO)$ which
controls the ``admissible'' $G$-sets $T$ for which equivariant
multiplications exist.  The commutative monoids in orthogonal
$G$-spectra correspond to the ``complete'' $\Ninfty$ operads which
permit all norms.

In this paper, we turn to the study of the equivariant symmetric monoidal
structure on categories of operadic modules associated to algebras
over particular $\Ninfty$ operads: the linear isometries operads
determined by a (possibly incomplete) $G$-universe $U$.  
Here the admissible sets for
$U$ will play a second role; for an $\aO$-algebra $R$, the admissible
sets determine additional structure on the underlying symmetric
monoidal category of $R$-modules.  Specifically, for each admissible
$G$-set $T$, we have a internal norm in the category of $R$-modules
for an $\aO$-algebra $R$.

In order to describe this structure, it is convenient to instead
consider the collection of categories of modules over $\iota_H^* R$,
where $H \subseteq G$ is a closed subgroup of $G$ and $\iota_H^*$ is
the forgetful functor.  The extra structure on the category of
$R$-modules then is encoded in functors
\[
\iota_H^* \colon \Mod_{\iota_K^* R} \to \Mod_{\iota_H^* R}
\qquad\textrm{and}\qquad
_{(\iota_K^* R)} N_{H}^{K} \colon \Mod_{\iota_H^*
R} \to \Mod_{\iota_K^* R}
\]
for $H \subset K \subset G$ that assemble into a kind of ``incomplete
Mackey functor'' of homotopical categories.  The internal norms arise
from the composite functors $N_H^G \iota^* H (-)$, extended to
arbitrary $G$-sets $T$ by decomposing $T$ into a disjoint union of
orbits $G/H$ and smashing together the corresponding composites.  The
compatibility conditions in particular express the double coset
formula.

More precisely, we have the following functors:

\begin{theorem}\label{thm:main1}
Let $G$ be a finite group and $U$ a $G$-universe.  Let $R$ be an
algebra in orthogonal $G$-spectra over $\sL_U$, the linear isometries
operad structured by $U$.  Then for each $H \subset G$ there exists
a symmetric monoidal model category $\aM_{\iota_H^* R}$ of $\iota_H^* 
R$-modules.  For each $H \subset K \subset G$ such that $K/H$ is an
admissible $K$-set for $U$, there exist homotopical functors
\[
_{(\iota_K^* R)} N_{H, \iota_H^* U}^{K, \iota_K^* U} \colon \aM_{\iota_H^* R} \to \aM_{\iota_K^*
R} \qquad\textrm{and}\qquad \iota_H^* \colon \aM_{\iota_K^*
R} \to \aM_{\iota_H^* R}.
\]
\end{theorem}

The internal norms now arise from these functors:

\begin{definition}
Let $G$ be a finite group and $U$ be a $G$-universe.  For an $H$-set
$T$, writing $T = H/K_1 \amalg H/K_2 \amalg \ldots \amalg H/K_m$,
define
\[
_{(\iota_H^* R)} N^T \colon \aM_{\iota_H^* R} \to \aM_{\iota_H^* R}
\]
by the formula
\begin{multline*}
_{(\iota_H^* R)} N^T X = \\
\left( _{(\iota_H^* R)} N_{K_1}^H \iota_{K_1}^*
X\right) \sma_{\iota_H^* R} \left( _{(\iota_H^* R)} N_{K_2}^H \iota_{K_2}^*
X\right) \sma_{\iota_H^* R} \ldots \sma_{\iota_H^* R} \left( _{(\iota_H^* R)} N_{K_m}^H \iota_{K_m}^* X\right).
\end{multline*}
More generally, define
\[
_R N^T \colon \aM_R \to \aM_R
\]
by the formula
\[
_R N^T X = G_+ \sma_H \left(_{(\iota_H^* R)} N^T X \right).
\]
\end{definition}

The equivariant symmetric monoidal structure on $\aM_R$ is encoded by the
following relations between the internal norms and the forgetful
functors:

\begin{theorem}\label{thm:gsymm}
Let $G$ be a finite group and $U$ be a $G$-universe.

\begin{enumerate}
\item For $H_1 \subseteq H_2 \subseteq H_3 \subseteq G$, there are
natural isomorphisms
\[
N_{H_2}^{H_3} N_{H_1}^{H_2} \cong
N_{H_1}^{H_3} \qquad\textrm{and}\qquad
\iota_{H_1}^* \iota_{H_2}^* \cong \iota_{H_1}^*
\]
that descend to the derived category,
\item For any $H$-sets $T_1$ and $T_2$, there are natural isomorphisms
$N^{T_1 \times T_2} X \cong N^{T_1} N^{T_2} X$ for each $X$ that descend to the
derived category when $T_1$ and $T_2$ are admissible, and
\item For an admissible $H$-set $T$, the derived
composite $\iota_K^* N^T M$ is naturally equivalent to $N^{\iota_K^*
T} \iota_K^* M$.
\end{enumerate}

The last of these relations is a version of the double coset
formula.
\end{theorem}

When $G = e$, the structure described by Theorem~\ref{thm:gsymm} is
simply the usual symmetric monoidal structure on orthogonal spectra;
the functors $N^T$ for a set $T$ are just the smash powers $X^{\sma |T|}$.
When $U$ is the complete universe, this structure is precisely the
$G$-symmetric monoidal structure on $R$-modules obtained by choosing a
model of $R$ that is a commutative monoid in orthogonal $G$-spectra.

Note that we have avoided trying to precisely formulate the notion of
an incomplete Mackey functor of homotopical categories here, choosing
instead to explicitly write out the structure and some of the
coherences.  However, if we are willing to pass to the homotopy
category, we can state the following result.

\begin{corollary}
Let $R$ be an algebra in orthogonal $G$-spectra over $\sL_U$.  Let
$\aB_{G,U}$ denote the bicategory of spans of the admissible sets for
$\sL_U$.  There exists a $2$-functor from $\aB_{G,U}$ to the
$2$-category of triangulated categories, exact functors, and natural
isomorphisms that takes an admissible set $G/H$ to $\aM_{\iota_H^*
R, \iota_H^* U}$.
\end{corollary}

However, the coherences necessary for the definition of an incomplete
Mackey functor at the level of homotopical categories is most easily
handled using the formalism of $\infty$-categories; we expect such a
treatment to come from the forthcoming work of~\cite{ClanBarwick}.
(See also~\cite{BohmannOsorno} for a treatment of equivariant
permutative categories from this kind of perspective.  A different
approach to equivariant permutative categories is described
in~\cite{GuillouMay}.)  

One of the applications of our work is the construction of strict
point-set models of $N_\infty$ ring spectra.  Specifically, let $\Sg$
be the equivariant sphere spectrum, regarded as an $\sL_U$ algebra.
Then we have the following straightforward consequence of the proof of
Theorem~\ref{thm:main1}.

\begin{corollary}
The category of commutative monoid objects in $\aM_{\Sg}$ is
equivalent to the category of $N_\infty$ algebras structured by
$\sL_U$.
\end{corollary}

More generally, for an $N_\infty$ algebra $R$, we obtain a description
of $N_\infty$ $R$-algebras.

\begin{corollary}
Let $R$ be an $N_\infty$ algebra structured by $\sL_U$.  The category
of commutative monoid objects in $\aM_R$ is equivalent to the category
of $N_\infty$ $R$-algebras structured by $\sL_U$.
\end{corollary}

These corollaries are particularly useful in the context of
equivariant Bousfield localization.  In their study of the
multiplicative properties of equivariant Bousfield localization, the
second author and Hopkins showed that localization of an $N_\infty$
ring spectrum can change the universe that structures the
multiplication~\cite{HillHopkinsLocalization}.
Specifically,~\cite[6.3]{HillHopkinsLocalization} shows that a
Bousfield localization $L$ of orthogonal $G$-spectra takes $\sL_U$
algebras to $\sL_U$ algebras precisely when the category of
$L$-acyclics is closed under norms for the indexing system determined
by $U$.  Therefore, we obtain the following result.

\begin{theorem}\label{thm:mainloc}
Let $A$ be a commutative monoid in $M_{\bS_U^G}$, where $\bS_U^G$
denotes the sphere spectrum regarded as an $\sL_U$ algebra in
orthogonal $G$-spectra.  Let $L$ be a Bousfield localization functor
with $L$-acyclics closed under norms specified by the indexing system
for a universe $U'$.  Suppose that $U''$ is a universe with
corresponding indexing system contained in the indexing system
obtained as the intersection of $U$ and $U'$.  Then $LR$ is a
commutative monoid object in $M_{\bS_{U''}^G}$.
\end{theorem}

In order to explain the restriction to $\Ninfty$ operads that can be
modeled as linear isometries operads, we need to explain the strategy
of proof for Theorem~\ref{thm:main1}.  Our approach is to adapt the
strategy of EKMM to study operadic multiplications on $G$-spectra.
Let $\Sp_G$ denote the category of orthogonal $G$-spectra on a
complete universe.  Fix a different (possibly incomplete) $G$-universe
$U$.  Then there is a monad $\bL_U$ on $\Sp_G$, specified by the
formula
\[
X \mapsto \sL(U,U)_+ \sma X,
\]
where $\sL(U,U)$ is the $G$-space of non-equivariant linear isometries
from $U$ to $U$ (with $G$ acting by conjugation).

The category $\Sp_G[\bL_U]$ of $\bL_U$-algebras has a model structure
that is Quillen equivalent to the standard model structures on
$\Sp_G$.  Moreover, it has a new symmetric monoidal product $\sma_U$
such that the underlying orthogonal $G$-spectrum of $X \sma_U Y$ is
equivalent to $X \sma Y$.  But now monoids and commutative monoids for
$\sma_U$ are precisely (non)-symmetric algebras for the $G$-linear
isometries operad for $U$.  Just as in the category of spectra, we can
restrict to the unital objects in $\Sp_G[\bL_U]$ to obtain a symmetric
monoidal category $\gspectra_U$.  All of these categories can be
equipped with symmetric monoidal model category structures.  Using
these symmetric monoidal model categories, we construct symmetric
monoidal module categories for an $\Ninfty$ ring $R$ structured by the
$G$-linear isometries operad for $U$.

We expect that Theorem~\ref{thm:main1} is true more generally for any
$\Ninfty$ operad, but it is difficult to obtain control on categories
of operadic modules over operads other than the linear isometries
operad using point-set techniques.  In fact, a substantial part of the
work of this paper involves verification of delicate point-set facts
about the linear isometries operad that are simply not true for an
arbitrary $N_\infty$ operad, just as in~\cite{EKMM}.  Unfortunately,
as we explain in~\cite[Theorem~4.24]{BlumbergHill}, there are
equivariant operads which arise from ``little disks'' constructions
that are not equivalent to equivariant linear isometries operads for
any universe.  Again, we expect that it is more tractable to handle
these sorts of homotopical questions in the $\infty$-categorical
setting; specifically, working with equivariant $\infty$-operads
structured over the nerve of distinguished subcategories of the
category of finite $G$-sets.

One benefit of our approach to Theorem~\ref{thm:main1} is that our
technical results about the equivariant linear isometries operad
validate the multiplicative theory of the equivariant version of EKMM
spectra.  Although~\cite[0.1]{EKMM} famously asserts that all of the
work of that volume holds mutatis mutandis when assuming that a
compact Lie group $G$ acts, verifying such a theorem requires some
subtle checks about the behavior of the linear isometries operad (most
notably Theorem~\ref{thm:mandell});
and~\cite{ElmendorfMay}, which amongst other endeavors attempts to
justify some of these properties, contains a critical error
(in~\cite[1.2]{ElmendorfMay}).  As such, our work in this paper
supports prior applications of the equivariant category of
$S$-modules, notably~\cite{GreenleesMay}.

Our interest in Theorem~\ref{thm:main1} comes in large part from
consequences for the foundations of equivariant derived algebraic
geometry.  
%
%
%
As explained above, localization of an $N_\infty$ ring
spectrum can change the universe that structures the
multiplication~\cite{HillHopkinsLocalization}.  This implies that
there is not necessarily a ``genuine'' affine scheme associated to a
commutative ring orthogonal $G$-spectrum when we work with the Zariski
topology.  Work of Nakaoka shows that something similar is true for
Tambara functors: there does not exist a sheaf of Tambara functors on
the Zariski site of a Tambara functor~\cite{Nakoaka}.

However, by restriction of structure, every equivariant commutative
ring spectrum $R$ is also an algebra over $\sL_{\mathbb R^{\infty}}$,
the linear isometries operad for a trivial universe.  Bousfield
localization always preserves the property of being an algebra over
$\sL_{\mathbb R^{\infty}}$, so in particular, we do have a sheaf of
such rings in the Zariski topology.  Therefore, using the work of this
paper we can define equivariant derived affine schemes (and then more
general derived schemes by gluing) in this fashion.  More generally,
Theorem~\ref{thm:mainloc} explains the situations when we can expect
more general affines.  We intend to return to the study of equivariant
derived schemes in a subsequent paper.

As a concrete example of this circle of ideas, let $\mathcal
X\to\mathcal Y$ be a Galois cover of stacks with Galois group $G$, and
let $\mathcal Y\to\mathcal M_{Ell}$ be an \'etale map to the moduli
stack of elliptic curves. We can evaluate the Goerss-Hopkins-Miller
sheaf of topological modular forms $\mathcal O^{top}$ on these \'etale
maps, producing commutative ring spectra and maps
\[
\mathrm{TMF}(\mathcal Y)\to\mathrm{TMF}(\mathcal X).
\]
The $G$-action on $\mathcal X$ gives a $G$-action on
$\mathrm{TMF}(\mathcal X)$, and we can then view this as a genuine
commutative equivariant ring spectrum by pushing forward to a complete
universe (see~\cite{HillMeier} for a related discussion).  We would
like to be able to understand the category of equivariant
$\mathrm{TMF}(\mathcal X)$-modules in algebro-geometric terms.  The
machinery presented in this paper is an essential tool in this
endeavor, making it possible to make sense of sheaves of modules on
the Zariski site.

\subsection*{Acknowledgments.}  The authors would like to thank Mike
Mandell for his assistance and Tony Elmendorf, John Greenlees, Mike
Hopkins, Magda Kedziorek, Peter May, and Brooke Shipley for many
helpful conversations.  This paper was improved by helpful comments by
an anonymous referee.  This project was made possible by the
hospitality of MSRI and the Hausdorff Research Institute for
Mathematics at the University of Bonn.

\section{Review of $\Ninfty$ operads}

In this section, we review the framework for describing equivariant
commutative ring spectra that we will work with in the paper.  We
refer the reader to~\cite{BlumbergHill} for a more detailed
discussion.

Let $G$ be a finite group and let $\gspectra$ denote the category of
orthogonal $G$-spectra structured by a complete universe and with
morphisms all (not necessarily equivariant) maps.  We will tacitly
suppress notation for the ``additive'' universe implicit in the
definition of $\gspectra$, as we are focused on multiplicative
phenomena.  Recall that the category $\gspectra$ is a complete and
cocomplete closed symmetric monoidal category under the smash product
$\sma$ with unit the equivariant sphere spectrum $\Sg$.  We will write
$F(-,-)$ for the internal mapping $G$-spectrum in $\gspectra$.  The
category $\gspectra$ is enriched over based $G$-spaces and has tensors
and cotensors; for $X$ an object of $\gspectra$, the tensor with a
based $G$-space $A$ is given by the smash product $A \sma X$ and the
cotensor by the function spectrum $F(A,X)$.

The enrichment of $\gspectra$ means that we can regard operads in
$G$-spaces as acting on objects of $\gspectra$ via the addition of a
$G$-fixed disjoint basepoint and the tensor.  Given a $G$-operad in
spaces, recall the following definition from~\cite[Definition
3.7]{BlumbergHill}.

\begin{definition}
An $\Ninfty$ operad is a $G$-operad $\aO$ such that
\begin{enumerate}
\item The space $\aO_0$ is $G$-contractible,
\item The action of $\Sigma_n$ on $\aO_n$ is free,
\item and $\aO_n$ is a universal space for a family $\aF_n(\aO)$ of
subgroups of $G \times \Sigma_n$ which contains all subgroups of the
form $H \times \{1\}$.
\end{enumerate}
\end{definition}

For any $\Ninfty$ operad $\aO$, there is an associated category
$\aO\mbox{-}\Alg$ of $\aO$-algebras in $\gspectra$.  We will be
particularly interested in the algebras associated to the $G$-linear
isometries operads.  Fix a possibly incomplete universe $U$ of
finite-dimensional $G$-representations; we adopt the standard
convention that $U$ contains a trivial representation and each of its
finite-dimensional subrepresentations infinitely often.  We do not
assume any relationship between $U$ and the ``additive'' universe that
arises in the definition of $\gspectra$.

\begin{definition}
The $G$-linear isometries operad $\sL_U$ has $n$th space $\aL_U(n)
= \sL(U^n,U)$ the $G$-space of non-equivariant linear isometries
$U^n \to U$ equipped with the conjugation action.  The distinguished
element $1 \in \aL_U(1)$ is the identity map and the operad structure
maps are induced by composition and direct sum.
\end{definition}

Recall from~\cite[Theorem~4.24]{BlumbergHill} that the $G$-linear isometries
operads do not always describe all of the possible $\Ninfty$
operads.  Nonetheless, they do capture many examples of interest, in
particular including the trivial and complete multiplicative
universes.

One of the major themes of our previous study of $\Ninfty$ operads was
that the essential structure encoded by an operad $\aO$ is the
collection of {\em admissible sets}.  We now review the relevant
definitions from~\cite[\S 3]{BlumbergHill}.

\begin{definition}
A {\emph{symmetric monoidal coefficient system}} is a contravariant
functor $\underline{\cC}$ from the orbit category of $G$ to the
category of symmetric monoidal categories and strong symmetric
monoidal functors.  The {\emph{value at H}} of a symmetric monoidal coefficient system
$\underline{\cC}$ is $\underline{\cC}(G/H)$, and will often be
denoted $\underline{\cC}(H)$.
\end{definition}

The most important example of a symmetric monoidal coefficient system
for us is the coefficient system of finite $G$-sets.

\begin{definition}
Let $\mSet$ be the symmetric monoidal coefficient system of finite
sets. The value at $H$ is $\Set^{H}$, the category of finite $H$-sets
and $H$-maps. The symmetric monoidal operation is disjoint union.
\end{definition}

We will associate to every $\Ninfty$ operad a subcoefficient system of
$\mSet$.  The operadic structure gives rise to additional structure on
the coefficient system.

\begin{definition}
We say that a full sub symmetric monoidal coefficient system $\cF$ of
$\mSet$ is {\emph{closed under self-induction}} if whenever
$H/K\in \cF(H)$ and $T\in \cF(K)$, $H\times_{K} T\in \cF(H)$.
\end{definition}

\begin{definition}
Let $\cC\subset\cD$ be a full subcategory. We say that $\cC$ is a
{\emph{truncation subcategory}} of $\cD$ if whenever $X\to Y$ is monic
in $\cD$ and $Y$ is in $\cC$, then $X$ is also in $\cC$.
A truncation sub coefficient system of a symmetric monoidal
coefficient system $\underline{\cD}$ is a sub coefficient system that
is levelwise a truncation subcategory.
\end{definition}

In particular, for finite $G$-sets, truncation subcategories are
subcategories that are closed under passage to subobjects and which
are closed under isomorphism.

\begin{definition}
An {\emph{indexing system}} is a truncation sub symmetric monoidal
coefficient system $\mF$ of $\mSet$ that contains all trivial sets and
is closed under self induction and Cartesian product.
\end{definition}

One of the main structural theorems about $\Ninfty$ operads
~\cite[4.17]{BlumbergHill} is that an $\Ninfty$ operad $\aO$
determines an indexing system of admissible sets.  This connection
arises from the standard observation that subgroups $\Gamma$ of
$G \times \Sigma_n$ such that $\Gamma\cap
(\{1\}\times\Sigma_{n})=\{1\}$ arise as the graphs of homomorphisms
$H \to \Sigma_n$, for some $H \subseteq G$.

\section{Point-set categories of modules over an $\Ninfty$
algebra}

In this section, we describe an approach to constructing categories of
modules over $\Ninfty$ algebras that proceeds via a rigidification
argument.  Of course, for any given $\Ninfty$ operad $\aO$ and an
$\aO$-algebra $R$ in orthogonal $G$-spectra, we can construct a
category of operadic modules over $R$.  However, experience in the
non-equivariant case teaches us that for practical work it is
extremely convenient to have rigid models of such categories that are
equipped with a symmetric monoidal smash product.

Specifically, for each linear isometries operad $\aO=\aL(U)$, we will
construct a symmetric monoidal structure on a category Quillen
equivalent to orthogonal $G$-spectra such that monoids and commutative
monoids correspond to $\aO$-algebras.  We describe how to produce such
a structure by adapting the techniques pioneered in the development of
the EKMM category of $S$-modules.  See also~\cite{Blumthesis, BCS, KM,
Spitzweck} for other categories in which this kind of approach has
been developed.  We can then define modules over an $\aO$-algebra $R$
in the evident fashion.

We are not able to rigidify algebras and modules over $\Ninfty$
operads which are not equivalent to equivariant linear isometries
operads.  Although in these cases we can give a homotopical
construction of the tensor product of operadic $\aO$-modules in terms
of the bar construction, we do not have good point-set control.

\subsection{The point-set theory of $\bL_U$-algebras in orthogonal spectra}

We begin by discussing the point-set details of the category of
algebras for a monad obtained from the first part of the equivariant
linear isometries operad.  Recall that $\gspectra$ denotes the
category of orthogonal $G$-spectra structured by a complete
$G$-universe.  Since the universe implicit in the definition of
$\gspectra$ does not play an essential role in what follows (once the
weak equivalences are fixed), we continue to suppress this choice from
the notation.

\begin{remark}
It is possible to carry out the work of this paper in the context of
an incomplete additive universe on $\gspectra$; the simplest case
arises when the additive and multiplicative universes are the same.
We leave this elaboration (and its attendant complications) to the
interested reader.
\end{remark}

Fix a (possibly incomplete) universe $U$.  Let
\[
\aL_U(1) = \aL(U,U)
\]
denote the $G$-space of linear isometries $U \to U$; i.e., the space
of non-equivariant linear isometries $U \to U$ equipped with the
conjugation action.  More generally, we write $\aL_U(n)$ to denote
$\aL(U^n, U)$, the $n$th space of the equivariant linear isometries
operad.

Since $\gspectra$ is tensored over based $G$-spaces, the formula
\[
X \mapsto \aL_U(1)_+ \sma X
\]
specifies a monad
$\bL_U \colon \gspectra \to \gspectra$.  The monadic structure maps are
induced by the identity element $\id_U \in \aL_U(1)$ and the composition
\[
\aL(U,U) \times \aL(U,U) \to \aL(U,U).
\]

\begin{definition}
Let $\gspectra[G][\bL_U]$ denote the category of $\bL_U$-algebras in
$\gspectra$.
\end{definition}

Since the monad $\bL_U$ has a right adjoint $F(\aL_U(1)_+,-)$, the
observation of~\cite[I.4.3]{EKMM} implies that this right adjoint
determines a comonad $\bL_U^{\sharp}$ such that the category of
coalgebras over $\bL_U^{\sharp}$ is equivalent to $\gspectra[G][\bL_U]$.
As a consequence we conclude the following result about the existence
of limits and colimits.

\begin{lemma}\label{lem:compco}
The category $\gspectra[G][\bL_U]$ is complete and cocomplete, with
limits and colimits created in $\gspectra$.  Similarly,
$\gspectra[G][\bL_U]$ has tensors and cotensors with based $G$-spaces;
the indexed colimits and limits are created in $\gspectra$.
\end{lemma}

The category $\gspectra[G][\bL_U]$ is equipped with mapping
$G$-spectra $F_{\gspectra[G][\bL_U]}(-,-)$ defined by the equalizer
\[
\xymatrix{
F(X,Y) \ar@<.5ex>[r] \ar@<-.5ex>[r] & F(\bL_U X, Y),
}
\]
where the maps are induced by the action $\bL_U X \to X$ and the
adjoint of the composite
\[
(\bL_U X) \sma F(X,Y) \cong \bL_U (X \sma F(X,Y)) \to \bL_U Y \to Y.
\]

Next, we note that any orthogonal $G$-spectrum can be given a trivial
$\gspectra[G][\bL_U]$ structure.  Specifically, in addition to the
free $\bL_U$-algebra functor
\[
\aL_U(1)_+ \sma (-) \colon \gspectra \to \gspectra[G][\bL_U],
\]
there is another functor
$p^* \colon \gspectra \to \gspectra[G][\bL_U]$ determined by the
unique projection map $p \colon \sL_U(1) \to *$; i.e., we can equip
any orthogonal $G$-spectrum $X$ with the trivial structure map
\[
\sL_U(1)_+ \sma X \to (*)_+ \sma X \cong X.
\]
We will be most interested in the sphere spectrum $\Sg$ regarded as an
$\bL_U$-algebra in this fashion.  The pullback functor is the right
adjoint of a functor $Q \colon \gspectra[G][\bL_U] \to \gspectra$
specified by the formula $QX = \Sg \sma_{\Sigma^{\infty}_+ \sL_U(1)}
X$.

We now define a closed weak symmetric monoidal structure on
$\gspectra[G][\bL_U]$ with unit $\Sg$.  (Recall that a weak symmetric
monoidal category has a product and a unit satisfying all of the
axioms of a symmetric monoidal category except that the unit map is
not required to be an isomorphism~\cite[II.7.1]{EKMM}.)

\begin{definition}
Let $X$, $Y$ be objects of $\gspectra[G][\bL_U]$.  We define the smash
product $\sma_{U}$ to be the coequalizer of the diagram
\[
\xymatrix{
(\sL_U(2) \times \sL_U(1) \times \sL_U(1))_+ \sma (X \sma
Y) \ar@<1ex>[r] \ar@<-1ex>[r] & \sL_U(2)_+ \sma (X \sma Y) \to
X \sma_U Y
}
\]
where the maps are specified by the actions of $\sL_U(1)_+$ on $X$ and
$Y$ and the right action of $\sL_U(1) \times \sL_U(1)$ on $\sL_U(2)$
via block sum and precomposition.

We will sometimes write this coequalizer using the notation
\[
\sL_U(2) \times_{\sL_U(1) \times \sL_U(1)} (X \sma Y).
\]
\end{definition}

Here the left action of $\sL_U(1)$ on $\sL_U(2)$ induces a left action
of $\sL_U(1)$ on $X \sma_U Y$ which endows it with the structure of an
$\bL_U$ algebra.  As an example, when $X = \bL_U A$ and $Y = \bL_U B$
are free $\bL_U$-algebras,
\begin{equation}\label{eq:freeiso}
X \sma_U Y \cong \sL_U(2)_+ \sma (A \sma B).
\end{equation}

Analogously, we have an internal function object in $\gspectra[G][\bL_U]$
that satisfies the usual adjunction.

\begin{definition}
Let $X, Y$ be objects of $\gspectra[G][\bL_U]$.  We define the mapping
$\bL_U$-spectrum $F_{\sL_U}(X,Y)$ to be the equalizer of the diagram
\[
\xymatrix{
F_{\gspectra[G][\bL_U]}(\sL_U(2)_+ \sma X, Y) \ar@<1ex>[r] \ar@<-1ex>[r]
& F_{\gspectra[G][\bL_U]}((\sL_U(2) \times \sL_U(1) \times \sL_U(1))_+ \sma
X, Y),
}
\]
where the maps are induced by the action of $\sL_U(1) \times \sL_U(1)$
on $\sL_U(2)$ by block sum and via the adjunction homeomorphism
\begin{multline*}
F_{\gspectra[G][\bL_U]}((\sL_U(2) \times \sL_U(1) \times \sL_U(1))_+ \sma X,
Y) \cong \\
F_{\gspectra[G][\bL_U]}((\sL_U(2) \times \sL(1))_+ \sma X,
F_{\gspectra[G][\bL_U]}(\sL(1)_+ \sma \Sg, Y))
\end{multline*}
along with the action $\sL_U(1)_+ \sma X \to X$ as well as the
coaction
\[
Y \to F_{\gspectra[G][\bL_U]}(\sL(1)_+ \sma \Sg,Y).
\]
\end{definition}

In what follows, we will repeatedly make use of the fact that for any
$n>0$ and admissible set $T$, we can choose a $G$-equivariant
homeomorphism $\mathbb R\{T\}\otimes U \to U$ (see Lemma~\ref{lem:unisum} in
Appendix~\ref{app:lin}).  We now establish the basic properties of
$\sma_U$.
 
\begin{theorem}\label{thm:assoc}
Let $X$, $Y$, and $Z$ be objects of $\gspectra[G][\sL_U]$.  There is a
natural commutativity isomorphism
\[
\tau \colon X \sma_U Y \to Y \sma_U X
\]
and a natural associativity isomorphism
\[
(X \sma_U Y) \sma_U Z \cong X \sma_U (Y \sma_U Z).
\]
More generally, there is a canonical natural isomorphism
\[
X_1 \sma_U \ldots \sma_U
X_k \cong \sL_U(k) \times_{\underbrace{(\sL_U(1) \times \ldots \times \sL_U
(1))}_{k}}
(X_1 \sma \ldots \sma X_k),
\]
where the left-hand side is associated in any order and the right-hand
side denotes the evident coequalizer generalizing the definition of
$\sma_U$.
\end{theorem}

\begin{proof}
Commutativity is essentially immediate (see~\cite[I.5.2]{EKMM}) and
associativity is a consequence of the equivariant analogue
of~\cite[I.5.4]{EKMM}, that is, the isomorphism
\[
\sL_U(i+j) \cong \sL_U(2) \times_{\sL_U(1) \times \sL_U(1)} \sL_U(i) \times \sL_U(j),
\]
which we prove as Lemma~\ref{lem:normhop}.  Associativity and the last formula now
follow from the arguments for~\cite[I.5.6]{EKMM}.  Specifically, we
have natural isomorphisms $\sL_U(1) \times_{\sL_U(1)} X \cong X$ for all $X$ in
$\gspectra[G][\bL_U]$, and therefore there are natural isomorphisms
\begin{align*}
X &\sma_U Y \sma_U Z \\
&\cong \sL_U(2) \times_{\sL_U(1) \times \sL_U(1)}
(\sL_U(2) \times_{\sL_U(1) \times \sL_U(1)} (X \sma Y)) \sma
(\sL_U(1) \times_{\sL_U(1)} Z) \\
&\cong (\sL_U(2) \times_{\sL_U(1) \times \sL_U(1)} \sL_U(2) \times \sL_U(1)) \times_{\sL_U(1) \times \sL_U(1) \times \sL_U(1)}
(X \sma Y \sma Z) \\
&\cong
\sL_U(3) \times_{\sL_U(1) \times \sL_U(1) \times \sL_U(1)} (X \sma
Y \sma Z).
\end{align*}
\end{proof}

Next, we construct the unit map, which is a consequence of the
equivariant analogue of a basic point-set property of spaces of linear
isometries; see Lemma~\ref{lem:equikriz}.

\begin{corollary}\label{corollary:unitisunit}
There is a natural isomorphism of $\bL_U$-spectra
\[
\lambda \colon \Sg \sma_U \Sg \cong \Sg
\]
such that $\lambda \tau = \lambda$.
\end{corollary}

The argument for~\cite[I.8.3]{EKMM} now generalizes without change to
the equivariant setting:

\begin{theorem}\label{thm:unit}
Let $X$ be an object of $\gspectra[G][\bL_U]$.  Then there exists a natural
map
\[
\psi \colon \Sg \sma_U X \to X
\]
which is compatible with the commutativity and associativity
isomorphisms.
\end{theorem}

Combining Theorems~\ref{thm:assoc} and~\ref{thm:unit}, we have proved
the following result.

\begin{theorem}
The category $\gspectra[G][\bL_U]$ is a closed weak symmetric monoidal
category with product $\sma_U$, unit $\Sg$, and function object
$F_{\sL_U}(-,-)$.
\end{theorem}

Just as in the setting of spaces~\cite{Blumthesis, BCS} and
spectra~\cite{EKMM}, we can actually work with the closed symmetric
monoidal category obtained by restricting to the unital objects.

\begin{definition}
Let $\gspectra_U$ denote the full subcategory of $\gspectra[G][\bL_U]$
consisting of those objects for which $\psi \colon \Sg \sma_U X \to X$
is an isomorphism.  For $X,Y$ objects in $\gspectra_U$, let $F_U(X,Y)$
denote $\Sg \sma_U F_{\sL_U}(X,Y)$ and abusively denote by $X \sma_U Y$ the
coequalizer regarded as an object of $\gspectra_U$.
\end{definition}

Corollary~\ref{corollary:unitisunit} implies that there is a functor
$\Sg \sma_U (-) \colon \gspectra[G][\bL_U] \to \gspectra_U$ which is the
left adjoint to the functor $F_{\sL_U}(\Sg,-)$ and the right adjoint to
the forgetful functor.  As a consequence, we can deduce the following
result.

\begin{proposition}
The category $\gspectra_U$ is complete and cocomplete.  Colimits are
created in $\gspectra[G][\bL_U]$ (and hence in $\gspectra$).  Limits are
formed by applying $\Sg \sma_U (-)$ to the limit in
$\gspectra[G][\bL_U]$.  Similarly, $\gspectra_U$ has tensors and
cotensors with based $G$-spaces.  Tensors are created in
$\gspectra[G][\bL_U]$ (and hence in $\gspectra$).  Cotensors are
formed by aplying $\Sg \sma_U (-)$ to the cotensor in
$\gspectra[G][\bL_U]$.
\end{proposition}

It is now straightforward to conclude the following result.

\begin{theorem}
The category $\gspectra_U$ is a closed symmetric monoidal category
with unit $\Sg$, product $\sma_U$, and function object $F_U(-,-)$.
\end{theorem}

\subsection{Point-set multiplicative change of universe functors}

A counterintuitive but useful fact about the category of orthogonal
$G$-spectra is that the point-set change of universe functors are
symmetric monoidal equivalences of categories.  In particular, for any
universe $U$, there is an equivalence of categories between
$G$-objects in the category of (non-equivariant) orthogonal spectra
and orthogonal $G$-spectra on $U$.  In this section, we explain the
corresponding result in the context of {\em multiplicative} change of
universe functors for the categories $\gspectra[G][\bL_U]$ as $U$
varies; the underlying (complete) additive universe that structures
$\gspectra[G]$ remains constant.

Let $U$ and $U'$ be $G$-universes, and denote by $\sL(U,U')$ the
$G$-space of non-equivariant linear isometries $U \to U'$, where $G$
acts by conjugation.  When $U = U'$, note that $\sL(U,U) = \sL_U(1)$.

\begin{definition}\label{defn:multichange}
Let $U$ and $U'$ be $G$-universes.  We define the functor
\[
\sL\sI_U^{U'} \colon \gspectra[G][\bL_U] \to \gspectra[G][\bL_{U'}]
\]
by setting $\sL\sI_U^{U'} X$ to be the coequalizer of the diagram
\[
\xymatrix{
\sL(U,U')_+ \sma \sL(U,U)_+ \sma X  \ar@<1ex>[r] \ar@<-1ex>[r] & \sL(U,U')_+ \sma X
}
\]
where the maps are determined by the action of $\sL_U(1)$ on $X$ and
the composition $\sL(U,U') \times \sL(U,U) \to \sL(U,U')$.  The action
of $\sL_{U'}(1)$ on $\sI_U^{U'} X$ is also induced by the composition
map $\sL(U',U') \times \sL(U,U') \to \sL(U,U')$.
\end{definition}

As explained in~\cite[1.3, 1.4]{ElmendorfMay}, we have the following
point-set result about the behavior of these functors.  We include the
proof here in order to make this paper more self-contained.

\begin{theorem}\label{thm:multichange}
Let $U$ and $U'$ be $G$-universes.  The functors $\sL\sI_{U}^{U'}$ and
$\sL\sI_{U'}^{U}$ are inverse equivalences of categories between
$\gspectra[G][\bL_U]$ and $\gspectra[G][\bL_{U'}]$.  Both functors are
strong symmetric monoidal.  As a consequence, the change of universe
functors descend to the categories $\gspectra_U$ and $\gspectra_{U'}$.
\end{theorem}

\begin{proof}
This result follows from the identification of the coequalizer
\[
\xymatrix{
\sL(U', U'') \times \sL(U',U') \times \sL(U,U') \ar@<1ex>[r] \ar@<-1ex>[r] &
\sL(U', U'') \times \sL(U,U')
}
\]
as $\sL(U,U'')$, for any universes $U$,$U'$, and
$U''$~\cite[2.2]{ElmendorfMay} and where the maps are all induced by
the composition $\gamma$.  Since coequalizers in $G$-spaces are
computed using the forgetful functor to spaces, it suffices to show
that this is a coequalizer diagram of non-equivarant spaces.  But in
this setting, the diagram is a split coequalizer.  The splitting is
constructed as follows.  Choose an isomorphism $s \colon U \to U'$ and
define
\[
h \colon \sL(U,U'') \to \sL(U', U'') \times \sL(U,U')
\]
and
\[
k \colon \sL(U', U'') \times \sL(U,U') \to \sL(U',
U'') \times \sL(U',U') \times \sL(U,U')
\]
via the formulas $h(f) = (f \circ s^{-1}, s)$ and $k(g',g) = (g',
g \circ s^{-1}, s)$.  Then $\gamma \circ h = \id$,
$(\id \times \gamma) \circ k = \id$, and $(\gamma \times \id) \circ k
= h \circ \gamma$.
\end{proof}

In particular, we have the following surprising corollary.

\begin{corollary}\label{cor:schwede}
Let $U$ be any $G$-universe.  The categories $\gspectra[G][\bL_U]$ and
$\gspectra_U$ are equivalent to the categories
$\gspectra[G][\bR^{\infty}]$ and $\gspectra_{\bR^{\infty}}$
respectively.
\end{corollary}

In our work in this paper, we will make critical use of this
equivalence to establish some point-set properties of our categories
$\gspectra[G][\bL_U]$ and $\gspectra_U$, notably about the
multiplicative norm and the fixed-point functors.  In fact, as pointed
out by an anonymous referee, we could simplify some of the work of the
previous section by using the fact that $\gspectra_{\bR^{\infty}}$ can
be described as a diagram category; the construction of the smash
product, colimits, and limits is then immediate from general results
about diagram categories and Corollary~\ref{cor:schwede}.

\begin{remark}
The additive version of this phenomenon was originally discovered in
the context of equivariant $\Gamma$-spaces by
Shimakawa~\cite{Shimakawa} and was proved for orthogonal $G$-spectra
in~\cite[\S V.1]{MM}.  In the multiplicative setting, the use of these
formulas to simplify the point-set theory for the equivariant stable
category is sketched in~\cite[XXIII.4]{MayAlaska}, in the context of
the equivariant version of EKMM spectra~\cite{EKMM}; this exposition
followed~\cite{ElmendorfMay}.

Although these facts were known to experts for a long time, the
observation has become prominent after its use in the definition of
the Hill-Hopkins-Ravenel multiplicative norm~\cite{HHR}; it is vastly
simpler to define the norm directly on $G$-objects and use the
universe only to study the homotopy theory.  Another important recent
application of these ideas comes from global equivariant homotopy
theory; this technique is essentially required to make the point-set
approach to global equivariant homotopy theory
tractable~\cite{Schwedeglobal}.  Likely motivated by this fact,
Schwede~\cite{SchwedeBook} has advocated for developing the
foundations of equivariant stable homotopy theory from this
perspective (although on the other hand see~\cite[V.1.9]{MM} for a
contrary view).

Nonetheless, we believe that despite Corollary~\ref{cor:schwede}, it
is conceptually clarifying in our work to keep track of the
multiplicative universe at the point-set level.  The issue is simply
that we have two universes in play, the universe structuring the
additive theory and the universe structuring the multiplicative
theory.  We believe that the approach outlined
in~\cite[XXIII.4]{MayAlaska} works best when there is only a single
universe; i.e., when the additive and multiplicative universe
coincide.  Moreover, when doing homotopical work, there is of course
no way to avoid incorporating the universe explicitly when writing
down formulas for fibrant replacement and (right) derived functors.
\end{remark}

\subsection{Rings and modules in $\gspectra[G][\bL_U]$ and $\gspectra_U$}

We now turn to the characterization of multiplicative objects in
$\gspectra[G][\bL_U]$ and $\gspectra_U$.  The key observation about
$\sma_U$ is that (in direct analogy with the non-equivariant case),
monoids for $\sma_U$ are algebras over the non-$\Sigma$ linear
isometries operad $\sL_U$ and commutative monoids for $\sma_U$ are
algebras over the linear isometries operad $\sL_U$.  More precisely,
let $\bT$ and $\bP$ denote the monads structuring associative and
commutative monoid objects in $\gspectra[G][\bL_U]$ respectively.
Concretely, for $X$ an object of $\gspectra[G][\bL_U]$,
\[
\bT X = \bigvee_{k \geq 0} \underbrace{X \sma_U \ldots \sma_U
X}_k \qquad \textrm{and} \qquad \bP X = \bigvee_{k \geq
0} \left(\underbrace{X \sma_U \ldots \sma_U X}_k \right) / \Sigma^k,
\]
where $X^0$ is defined to be $\Sg$.

Monadic algebras over $\bT$ and $\bP$ in $\gspectra_U$ are simply
algebras in $\gspectra[G][\bL_U]$ that are unital; there are functors
\[
\Sg \sma_U (-) \colon (\gspectra[G][\bL_U])[\bT] \to \gspectra_U[\bT]
\]
and
\[
\Sg \sma_U (-) \colon (\gspectra[G][\bL_U])[\bP] \to \gspectra_U[\bP].
\]

The next result connects the categories $(\gspectra[G][\bL_U])[\bT]$ and
$(\gspectra[G][\bL_U])[\bP]$ of monadic algebras to categories of
operadic $\Ninfty$ algebras~\cite{BlumbergHill}.

\begin{theorem}\label{thm:commninf}
The category $(\gspectra[G][\bL_U])[\bT]$ is isomorphic to the category of
non-$\Sigma$ $\sL_U$-algebras in $\gspectra$.  The category
$(\gspectra[G][\bL_U][\bP])$ is isomorphic to the category of
$\sL_U$-algebras in $\gspectra$.
\end{theorem}

\begin{proof}
The argument is the same as the proof of~\cite[II.4.6]{EKMM}, using
the isomorphism of equation~\eqref{eq:freeiso} levelwise.
\end{proof}

In light of the previous theorem, we will refer to monoids and
commutative monoids in $\gspectra[G][\bL_U]$ and $\gspectra$ as
associative and commutative $\Ninfty$ ring orthogonal $G$-spectra,
respectively.

Next, the arguments of~\cite[II.7]{EKMM} extend to prove the
following:

\begin{theorem}
The categories $(\gspectra[G][\bL_U])[\bT]$, $(\gspectra[G][\bL_U])[\bP]$,
$\gspectra_U[\bT]$, and $\gspectra_U[\bP]$ are complete and
cocomplete, with limits created in $\gspectra$.  The categories
$(\gspectra[G][\bL_U])[\bT]$ and $\gspectra_U[\bT]$ are tensored and
cotensored over based $G$-spaces, with cotensors created in
$\gspectra[G][\bL_U]$ and $\gspectra_U$ respectively.  The categories
$(\gspectra[G][\bL_U])[\bP]$ and $\gspectra_U[\bP]$ are tensored and
cotensored over unbased $G$-spaces, with cotensors created in
$\gspectra[G][\bL_U]$ and $\gspectra_U$ respectively (regarding these
categories as cotensored over unbased spaces via the functor that
adjoins a disjoint $G$-fixed basepoint).
\end{theorem}

As an aside, we note the following standard observation, which follows
as usual simply by checking the universal property.

\begin{lemma}
The symmetric monoidal product $\sma_U$ is the coproduct on
$\gspectra_U[\bP]$.
\end{lemma}

Finally, for any monoid or commutative monoid $R$, there are
associated categories of (left) $R$-modules in $\gspectra[\bL_U]$ and
$\gspectra_U$.  Since the theory is cleanest in the case of
$\gspectra_U$, we focus on the unital setting in the following
discussion.  The multiplication and unit maps for $R$ give the functor
$R \sma_U (-)$ the structure of a monad on $\gspectra_U$.

\begin{definition}
Let $R$ be an object in $\gspectra_U[\bT]$ or $\gspectra_U[\bP]$.  The
category $\aM_{R,U}$ of $R$-modules in $\gspectra_U[\bP]$ is the
category of algebras for the monad $R \sma_U (-)$ in $\gspectra_U$.
\end{definition}

Such categories of $R$-modules are complete and cocomplete, with
limits and colimits created in $\gspectra_U$.  When $R$ is
commutative, the category of $R$-modules is closed symmetric monoidal with
unit $R$ and product $X \sma_{R,U} Y$ defined as the coequalizer of
the diagram
\[
\xymatrix{
X \sma_U R \sma_U Y \ar@<1ex>[r] \ar@<-1ex>[r] & X \sma_U Y }
\]
where the maps are induced by the right action of $R$ on $X$ via the
symmetry isomorphism and the left action of $R$ on $Y$.  The function
object is defined as the equalizer of the diagram
\[
\xymatrix{
F_U(X,Y) \ar@<1ex>[r] \ar@<-1ex>[r] & F_U(R \sma_U X, Y) 
}
\]
where the maps are induced by the action of $R$ on $X$ and the adjoint
of the composite
\[
R \sma_U X \sma_U F_U (X,Y) \to R \sma_U Y \to Y.
\]
There are also the evident categories of $R$-algebras and commutative
$R$-algebras. 

\begin{definition}
Let $R$ be an object in $\gspectra_U[\bP]$.  Abusively denote by $\bT$
and $\bP$ the monads in $\aM_{R,U}$ that structure monoids and
commutative monoids.  We refer to the categories $\aM_{R,U}[\bT]$ and
$\aM_{R,U}[\bP]$ as the categories of $R$-algebras and commutative
$R$-algebras respectively.
\end{definition}

\subsection{Change of group and fixed-point functors}\label{sec:change}

In this section, we study change-of-group and fixed-point functors in
the context of the categories $\gspectra[G][\bL_U]$ and
$\gspectra_U$.  If we are content to ignore the monoidal structure,
the point-set theory of the change of group and fixed-point functors
is the same as for $\gspectra[G]$.  The interaction of these functors
with the action of $\sL_U(1)$ is more subtle.  Our discussion relies
on observations from~\cite[\S VI.1]{MM}.

Let $\iota_H \colon H \to G$ be the inclusion of a subgroup.
Denote by $WH$ the quotient $NH/H$, where $NH$ is the normalizer of
$H$ in $G$.  For $X$ an object of $\gspectra$, there is a
homeomorphism
\[
\iota^*_H \bL_U X \cong \bL_{(\iota^*_H U)} (\iota^*_H X).
\]
This homeomorphism is easily seen to be compatible with the monad
structure, and so we obtain a functor
\[
\iota^*_H \colon \gspectra[G][\bL_U] \to \gspectra[H][\bL_{(\iota^*_H
U)}],
\]
where the additive universe on $\gspectra[H]$ here is $\iota^*$
applied to the complete universe structuring $\gspectra[G]$.
Analogously, for $Y$ an object of $\gspectra[H]$, we have a
homeomorphism
\[
G_+ \sma_H \bL_{(\iota^*_H U)} Y \cong \bL_U (G_+ \sma_H Y)
\]
that is compatible with the monad structure, producing a functor
\[
G_+ \sma_H (-) \colon \gspectra[H][\bL_{(\iota^*_H U)}] \to \gspectra[G][\bL_U]
\]
that is the left adjoint to $\iota^*_H$.  Finally, there is also a
homeomorphism
\[
F_H(G, \bL^{\sharp}_{(\iota^*_H U)} Y) \cong \bL^{\sharp} F_H(G, Y)
\]
(here recall that the comonad $\bL^{\sharp}$ is described just prior
to the proof of Lemma~\ref{lem:compco}) that is compatible with the
comonad structure and thus produces the right adjoint
\[
F_H(G, -) \colon \gspectra[H][\bL_{(\iota^*_H U)}] \to \gspectra[G][\bL_U]
\]
to $\iota^*_H$.

Furthermore, all of these functors are compatible with the functors
creating the unital objects, and so descend to functors
$\iota^*_H \colon \gspectra_U \to \gspectra[H]_{(\iota^*_H U)}$ and
the attendant left and right adjoints.

Finally, it is evident that $\iota^*_H$ is symmetric monoidal and so
it restricts to categories of monoids and commutative monoids.

\begin{proposition}
Let $H$ be a subgroup of $G$.  Then there are forgetful functors
\[
\iota_H^* \colon (\gspectra[G][\bL_U])[\bT] \to (\gspectra[H][\bL_{\iota_H^* U}])[\bT] \qquad\textrm{and}\qquad \iota_H^* \colon \gspectra_U[\bT] \to \gspectra_{\iota_H^* U}[\bT]
\]
and
\[
\iota_H^* \colon (\gspectra[G][\bL_U])[\bP] \to (\gspectra[H][\bL_{\iota_H^* U}])[\bP] \qquad\textrm{and}\qquad \iota_H^* \colon \gspectra_U[\bP] \to \gspectra_{\iota_H^* U}[\bP].
\]
\end{proposition}

Next, we turn to the question of the categorical fixed points.  Our
definition is built from the categorical fixed point functor $(-)^H$
on $\gspectra$~\cite[V.3.9]{MM}.

\begin{theorem}
Let $H$ be a subgroup of $G$.  Then the categorical $H$-fixed point
functor on $\gspectra$ induces a lax monoidal categorical
$H$-fixed point functor 
\[
(-)^H \colon \gspectra[G][\bL_U] \to \gspectra[WH][\bL_{U^H}]
\]
specified (in mild abuse of notation) by the formula
\[
X^H = (\sL\sI_U^{U^H} X)^H,
\]
where the $(-)^H$ on the righthand side denotes the categorical
fixed points in $\gspectra$.  The fixed point functor has an op-lax
symmetric monoidal left adjoint 
\[
\epsilon_H^* \colon \gspectra[WH][\bL_{U^H}] \to \gspectra[G][\bL_{U}].
\]
which assigns to a $WH$-spectrum $X$ the $G$-spectrum obtained by
pulling back along the quotient $NH \to WH$, changing (additive)
universe, and inducing up to $G$, and changing multiplicative
universe.  When $H$ is normal, the left adjoint is strong symmetric
monoidal.
\end{theorem}

\begin{proof}
Since
\[
(\sL(U^H, U^H)_+ \sma
X)^H \cong \sL(U^H,U^H)_+ \sma X^H,
\]
the categorical $H$-fixed point functor restricts to a functor
\[
\gspectra[G][\bL_{U^H}] \to \gspectra[WH][\bL_{U^H}].
\]
Analogously,
\[
\epsilon_H^* (\sL(U,U)_+ \sma Y) \cong \sL(U,U)_+ \sma \epsilon_H^* Y,
\]
which implies that $\epsilon_H^*$ restricts to a functor from
$\gspectra[WH][\bL_{U^H}]$ to
$\gspectra[G][\bL_{U}]$.

Next, we consider the interaction of $(-)^H$ with the monoidal
structure.  Since the action of $H$ on
$\sL(U^H,U^H)$ is trivial and $(-)^H$ is lax
monoidal on $\gspectra$~\cite[V.3.8]{MM}, for $X$ and $Y$ in
$\gspectra[G][\bL_U]$ and $H \subseteq G$ we have a natural map
\[
(\sL_{U^H}(2) \times_{\sL_{U^H}(1) \times \sL_{U^H}(1)})_+
\sma X^H \sma Y^H \to
(\sL_{U^H}(2) \times_{\sL_{U^H}(1) \times \sL_{U^H}(1)})_+
\sma (X \sma Y)^H
\]
which lands in the fixed-points
\[
\left((\sL_{U^H}(2) \times_{\sL_{U^H}(1) \times \sL_{U^H}(1)})_+
\sma (X \sma Y)\right)^H,
\]
and so we deduce that $(-)^H$ is a lax symmetric monoidal functor
\[
\gspectra[G][\bL_U] \to \gspectra[WH][\bL_{U^H}].
\]
Finally, when $H$ is normal, the left adjoint is strong symmetric
monoidal since the pullback and additive change of universe are. 
\end{proof}

The situation for the geometric fixed point functor is analogous;
again, we construct $\Phi^H$ on $\gspectra[G][\bL_U]$ by considering
the composite $\Phi^H (\sL\sI_U^{U^H} X)$.

\begin{theorem}
Let $H$ be a subgroup of $G$.  Then there is a lax symmetric
monoidal geometric $H$-fixed point functor
\[
\Phi^H \colon \gspectra[G][\bL_U] \to \gspectra[WH][\bL_{U^H}].
\]
\end{theorem}

\begin{proof}
The compatibility of the geometric fixed points functor with
$\sL_U(1)$ action is clear.  Next, once again the fact that the
actions of $H$ on $\sL_{U^H}(2)$ and $\sL_{U^H}(1)$ are trivial and
the fact that $\Phi^H$ is lax symmetric monoidal on $\gspectra$
implies that it is lax symmetric monoidal on $\gspectra[G][\bL_U]$.
\end{proof}

\subsection{The point-set theory of the norm}\label{sec:pointnorm}

In this subsection, we construct multiplicative norm functors in the
sense of~\cite{HHR} on the categories $\gspectra[G][\bL_U]$,
$\gspectra_U$, and $\aM_{R,U}$ for $R$ a commutative algebra in
$\gspectra_U$.  Fix a subgroup $H \subseteq G$ and let $\widehat{U}$
denote an $H$-universe.  The norm functor
$N_H^G \colon \gspectra[H] \to \gspectra[G]$ is strong symmetric
monoidal and so there is a natural isomorphism
\[
N_H^G \big(\sL_{\widehat{U}}(1)_+ \sma X\big) \cong
F_H(G, \sL_{\widehat{U}}(1))_+ \sma N_H^G X.
\]
This leads to the following definition (compare with
Theorem~\ref{thm:assoc}).

\begin{definition}\label{defn:internalabsnorm}
We define the functor
\[
N_{H,\widehat{U}}^{G,U} \colon \gspectra[H][\bL_{\widehat{U}}] \to \gspectra[G][\bL_U]
\]
on objects $X$ via the coequalizer of the diagram
\[
\xymatrix{
\sL(\Ind_H^G \widehat{U}, U)_+ \sma F_H(G, \sL_{\widehat{U}}(1))_+ \sma N_H^G X
\ar@<1ex>[r] \ar@<-1ex>[r] &
\sL(\Ind_H^G \widehat{U}, U)_+ \sma N_H^G X,
}
\]
where the right action of $U$ on $\sL(\Ind_H^G \widehat{U}, U)$
provides the structure of an $\bL_U$ algebra.
\end{definition}

In the coequalizer, the other map is specified by the action of
$F_{H}(G,\sL_{\widehat{U}}(1))$ on $\sL(\Ind_{H}^{G}\widehat{U},U)$
via the map of monoids
\[
I_{(-)}\colon F_{H}(G,\sL_{\widehat{U}}(1))\to \sL_{\Ind_{H}^{G}\widehat{U}}(1)
\]
given by
\[
f\mapsto I_{f}=\big(g\otimes u\mapsto g\otimes f(g)(u)\big),
\]
the target of which is underlain by the orthogonal sum of isometries
and hence is an isometry.

There is an alternate characterization of $N^{G,U}_{H, \iota_H^* U}$
which can be given using the multiplicative change of universe
functors.  

\begin{lemma}\label{lem:schnorm}
There is a natural isomorphism
\[
N_{H,\iota_H^* U}^{G,U} X \cong
\sL\sI_{\bR^{\infty}}^{U}
\left(N^{G,\bR^{\infty}}_{H, \bR^{\infty}} (\sL\sI_{\iota_H^* U}^{\bR^{\infty}} X)\right)
\]
\end{lemma}

\begin{proof}
To establish the identification, we expand the right-hand side,
writing $\bR$ in place of $\bR^{\infty}$ for concision:
\begin{align*}
&\sL(\bR,U) \times_{\sL(\bR, \bR)}
N^{G, \bR}_{H, \bR} \sL(\iota_H^*
U, \bR) \times_{\sL(\iota_H^* U, \iota_H^* U)} X \\
&\cong \sL(\bR,U) \times_{\sL_{\bR}(1)}
\left(
\sL(\Ind_{H}^{G} \bR, \bR) \times_{F_H(G, \sL_{\bR}(1))}
N_H^G \left( \sL(\iota_H^*
U, \bR) \times_{\sL_{\iota_H^* U}(1)} X
\right)\right) \\
&\cong \sL(\Ind_{H}^{G} \bR, U) \times_{F_H(G, \sL(\bR, \bR))}
N_H^G \sL(\iota_H^*
U, \bR) \times_{\sL(\iota_H^* U, \iota_H^* U)} X \\
&\cong \sL(\Ind_{H}^{G} \bR, U) \times_{F_H(G, \sL(\bR, \bR))}
F_H(G,\sL(\iota_H^*
U, \bR)) \times_{F_H(G,\sL(\iota_H^* U, \iota_H^* U))} N_H^G
X \\
&\cong \sL(\Ind_{H}^{G} \iota_H^* U, U) \times_{F_H(G,\sL(\iota_H^*
U, \iota_H^* U))} N_H^G X \\
& \cong N^{G,U}_{H,\iota_H^*} X.
\end{align*}

In these expansions, note that we use the fact that the norm preserves
reflexive coequalizers~\cite[A.54]{HHR}.
\end{proof}

We now show that norm is strong symmetric monoidal.  This can be done
using Lemma~\ref{lem:schnorm}, but it is convenient in the homotopical
analysis to give a slightly more expansive proof that involves a bit
more work with the linear isometries operad, also given in the
Appendix~\ref{app:lin}.  (In contrast, compare the proof of
Theorem~\ref{thm:adj} below.)

\begin{theorem}\label{thm:normstrong}
The functor $N_{H, \widehat{U}}^{G,U}$ is strong symmetric monoidal.
\end{theorem}

\begin{proof}
By Lemma~\ref{lem:normkriz}, we see that $N_{H,\widehat{U}}^{G,U}$
preserves the unit.  We now compare
$N_{H,\widehat{U}}^{G,U}(X \sma_{\widehat{U}} Y)$ and
$(N_{H,\widehat{U}}^{G,U} X) \sma_U (N_{H,\widehat{U}}^{G,U} Y)$ by
direct computation.  By definition, we have
\begin{multline*}
N_{H,\widehat{U}}^{G,U}(X \sma_{\widehat{U}} Y) = \\
\sL(\Ind_H^G \widehat{U},
U) \times_{F_H(G,\sL_{\widehat{U}}(1))} \left( N_H^G
(\sL_{\widehat{U}}(2) \times_{\sL_{\widehat{U}}(1) \times \sL_{\widehat{U}}(1)}
(X \sma Y) \right).
\end{multline*}
Since the norm functor commutes with reflexive coequalizers and is
symmetric monoidal as a functor on orthogonal $H$-spectra, this is
isomorphic to
\begin{multline*}
\sL(\Ind_H^G \widehat{U},
U) \times_{F_H(G,\sL_{\widehat{U}}(1))} \\
\left( F_H(G,\sL_{\widehat{U}}(2)) \times_{F_H(G,\sL_{\widehat{U}}(1)) \times F_H(G,\sL_{\widehat{U}}(1))}
(N_H^G X \sma N_H^G Y) \right).
\end{multline*}
Applying Corollary~\ref{cor:corhop2}, we rewrite this as
\[
\sL(\Ind_H^G \widehat{U} \oplus \Ind_H^G \widehat{U},
U) \times_{F_H(G, \sL_{\widehat{U}}(1)) \times
F_H(G, \sL_{\widehat{U}}(1))} (N_H^G X \sma N_H^G Y).
\]

On the other hand, writing out $(N_{H,\widehat{U}}^{G,U} X) \sma_U
(N_{H,\widehat{U}}^{G,U} Y)$ we have
\begin{align*}
\sL&_U(2) \times_{\sL_U(1) \times \sL_U(1)} \\
&\left((\sL(\Ind_H^G \widehat{U}, U) \times_{F_H(G,\sL_{\widehat{U}}(1))}
N_H^G X) \sma
(\sL(\Ind_H^G \widehat{U}, U) \times_{F_H(G,\sL_{\widehat{U}}(1))}
N_H^G Y) \right).
\end{align*}
Applying Corollary~\ref{cor:corhop}, we can rewrite this as
\[
\sL(\Ind_H^G \widehat{U} \oplus \Ind_H^G \widehat{U},
U) \times_{F_H(G, \sL_{\widehat{U}}(1)) \times F_H(G, \sL_{\widehat{U}}(1))} (N_H^G X \sma
N_H^G Y).
\]
Finally, the naturality of the isomorphisms above make it clear that
the pentagon identities hold.
\end{proof}

As a consequence of Theorem~\ref{thm:normstrong}, we have the
following corllary.

\begin{corollary}
The functor
$N_{H,\widehat{U}}^{G,U}$ restricts to a functor
\[
N_{H,\widehat{U}}^{G,U} \colon \gspectra[H]_{\widehat{U}} \to \gspectra_U
\]
which we abusively refer to with the same notation.
\end{corollary}

\begin{remark}
Lemma~\ref{lem:schnorm} can now be interpreted as the statement that
the norm can be described as the indexed product on
$\gspectra[H]_{\bR^{\infty}}$; this makes it clear that the norm is
functorial in both the group and the input spectrum.  
\end{remark}

We now turn to establish the following adjunction on commutative ring
objects.  We will be predominantly interested in the case where
$\widehat{U} = \iota_H^* U$, as this is relevant for describing the
equivariant symmetric monoidal structures on the categories
$\aM_{R,U}$.

\begin{theorem}\label{thm:adj}
There are adjoint pairs with left adjoints
\[
N_{H,\iota_H^* U}^{G,U} \colon (\gspectra[H][\bL_{\iota_H^*
U}])[\bP] \to (\gspectra[G][\bL_{U}])[\bP] \quad\textrm{and}\quad
N_{H,\iota_H^* U}^{G,U} \colon (\gspectra[H]_{\iota_H^* U})[\bP] \to \gspectra_U[\bP]
\]
and right adjoints
\[
\iota_H^* \colon (\gspectra[G][\bL_{U}])[\bP] \to (\gspectra[H][\bL_{\iota_H^*
U}])[\bP]
\quad\textrm{and}\quad
\iota_H^* \colon \gspectra_U[\bP] \to \gspectra[H]_{\iota_H^* U}[\bP]
\]
respectively.
\end{theorem}

\begin{proof}
First, observe that the conclusion of the theorem follows immediately
when $U = \bR^{\infty}$: since for any $G$ the category
$\gspectra[G][\bL_{\bR^\infty}]$ is equivalent to the category of
$G$-objects in $\mathcal{S}[\bL_{\bR^\infty}]$, we can
apply~\cite[A.56]{HHR}.  We now use the alternate characterization of
the norm from Lemma~\ref{lem:schnorm} and the fact that by
Theorem~\ref{thm:multichange} the change of universe functors are
symmetric monoidal equivalences of categories.
\end{proof}

An immediate corollary of Theorem~\ref{thm:adj} is that commutative
ring objects have an ``internal norm'' map arising from the counit of
the adjunction.

\begin{corollary}
Let $R$ be an object in $(\gspectra[G][\bL_U])[\bP]$ or
$\gspectra_U[\bP]$.  Then there is a natural map
\[
N_{H,\iota_H U}^{G,U} \iota_H^* R \to R.
\]
\end{corollary}

Using the counit of the adjunction of Theorem~\ref{thm:adj} and the
absolute norm functor described in
Definition~\ref{defn:internalabsnorm}, we can express the $R$-relative
norm for a commutative ring object $R$ in $\gspectra_U$ using
base-change:

\begin{definition}
Let $R$ be an object in $\gspectra_U[\bP]$.
We define the functor
\[
_R N_{H, \iota_H^* U}^{G,U} \colon \aM_{\iota_H^* R, \iota_H^*
U} \to \aM_{R,U}
\]
via the formula
\[
X \mapsto R \sma_{N_{H, \iota_H^* U}^{G,U} R} N_{H, \iota_H^* U}^{G,U} X,
\]
where the coequalizer is over the counit map and the map induced by
the action of $R$ on $X$.
\end{definition}

It is clear from the definition and Theorem~\ref{thm:normstrong} that
the $R$-relative norm is also a strong symmetric monoidal functor.

\section{Homotopical categories of modules over an $N_\infty$ algebra}

In this section, we describe model structures on the categories
$\gspectra[G][\bL_U]$, $\gspectra_U$, and categories of algebras and
modules over an algebra.  The main goal of our efforts is to describe
the derived functors of the norm and forgetful functors as a prelude
to the construction of the equivariant symmetric monoidal structure.

\subsection{The homotopical theory of $\gspectra[G][\bL_U]$ and $\gspectra_U$}

We begin by quickly reviewing some of the less commonly-used
terminology from the theory of model categories that we will employ in
the statements of results below.
Recall from~\cite[5.9]{mmss} that a cofibrantly generated topological
model structure is compactly generated if the domains of the
generating cofibrations and acyclic cofibrations are compact and
satisfy the ``Cofibration Hypothesis''~\cite[5.3]{mmss}.  Let $\aC$ be
a complete and cocomplete topologically enriched category.  An
$h$-cofibration in $\aC$ is a map that is the analogue of a Hurewicz
cofibration; i.e., a map $X \to Y$ such that the induced map $Y \cup_X
(X \otimes I) \to Y \otimes I$ has a retraction.  The Cofibration
Hypothesis for a set of maps $I$ in a model category $\aA$ equipped
with a forgetful functor $\aA \to \aC$ specifies that the following
two conditions are satisfied.

\begin{enumerate}
\item For a coproduct $A \to B$ of maps in $I$, in any pushout
\[
\xymatrix{
A \ar[r] \ar[d] & X \ar[d] \\
B \ar[r] & Y
}
\]
in $\aA$, the cobase change $X \to Y$ is an $h$-cofibration in $\aC$.
\item Given a sequential colimit in $\aA$ along maps that are
$h$-cofibrations in $\aC$, the colimit is $\aA$ is equal to the
colimit in $\aC$.
\end{enumerate}

In order to be able to apply Bousfield localization, it is convenient
to add the requirements that:
\begin{enumerate}
\item The domains of the generating acyclic
cofibrations are small with respect to the generating cofibrations,
\item and the cofibrations are effective monomorphisms.
\end{enumerate}

A compactly generated model category that satisfies these additional
conditions is cellular~\cite[12.1.1]{Hirschhorn}, and so admits left
Bousfield localizations very generally.  In mild abuse of terminology,
we will use the term compactly generated to refer to a compactly
generated model category that is cellular in this paper.

A model category is $G$-topological if it is enriched over $G$-spaces
and satisfies the analogue of Quillen's SM7~\cite[III.1.14]{MM}.
There is an evident $G$-equivariant version of the Cofibration
Hypothesis.  The building block for our work in this section is the
complete model structure on $\gspectra$~\cite[B.63]{HHR}.  (Note that
although the cited reference refers to the positive complete model
structure, the existence of the complete model structure is clear.)
Recall that the complete model structure has generating cofibrations
given by the set of maps
\[
\{G_+ \sma_H S^{-V} \sma S^{n-1}_+ \to G_+ \sma_H S^{-V} \sma D^{n}_+\},
\]
where $V$ varies over the (additive) universe, $n \geq 0$, and
$H \subseteq G$.

\begin{lemma}
The complete model structure is a compactly generated $G$-topological
model structure.
\end{lemma}

\begin{proof}
The discussion proving~\cite[B.63]{HHR} establishes that the complete model
structure is cofibrantly generated.  Since the generating cofibrations
in the complete model structure are $h$-cofibrations~\cite[B.64]{HHR},
it is straightforward to see that the complete model structure
satisfies the Cofibration Hypothesis.  Finally, the cofibrations are
effective monomorphisms since $S^{n-1}_+ \to D^{n}_+$ is for all
$n \geq 0$, and the compactness criterion for the domain of the
generating acyclics is clearly satisfied.
\end{proof}

\begin{theorem}
The category $\gspectra[G][\bL_U]$ is a compactly generated weak
symmetric monoidal proper $G$-topological model category in which the weak
equivalences and fibrations are detected by the forgetful functor
$U \colon \gspectra[G][\bL_U] \to \gspectra$.
\end{theorem}

\begin{proof}
The monad $\bL_U$ evidently satisfies the hypotheses of (the
equivariant analogue of)~\cite[5.13]{mmss}, and so we can conclude
that there is a compactly generated $G$-topological model structure on
$\gspectra[G][\bL_U]$.  The proof of the unit axiom follows from the
equivariant analogue of from~\cite[XI.3.1]{EKMM}, which holds by the
same proof as in the non-equivariant case.  To check the monoid axiom,
observe that it suffices to check on the generating (acyclic)
cofibrations, and since these are obtained by $\sL_U(1) \sma_+ (-)$
applied to generating (acyclic) cofibrations of $\gspectra$, the
result holds since it does in $\gspectra$.  Finally, it is clear that
$\gspectra[G][\bL_U]$ is proper.
\end{proof}

By construction, the adjoint pair $(\bL_U, U)$ is a Quillen
adjunction.  Since $\sL_U(1)$ is $G$-contractible, we can conclude
that this pair induces a Quillen equivalence between $\gspectra[G][\bL_U]$
and $\gspectra$.

\begin{proposition}
The adjoint pair $(\bL_U,U)$ forms a Quillen equivalence between
$\gspectra[G][\bL_U]$ and $\gspectra$.
\end{proposition}

Although $\bL_U$ is not strong symmetric monoidal,
it is close: as a consequence of Lemma~\ref{lem:unisum}, there is a
homeomorphism
\[
\bL_U X \sma_U \bL_U Y \cong \bL_U (X \sma Y)
\]
and more generally homeomorphisms
\[
\bL_U X_1 \sma_U \ldots \sma_U \bL_U X_k \cong \bL_U
(X_1 \sma \ldots \sma X_k).
\]
(The failure of $\bL_U$ to be strong symmetric monoidal is a
consequence of the fact that these homeomorphisms ultimately depend on
choices of homeomorphisms $U^k \to U$.)
On the other hand, the functor $Q(-)
= \Sg \sma_{\Sigma^{\infty}_+ \sL_U(1)} (-)$ is strong symmetric
monoidal~\cite[4.14]{BCS}.  As a consequence, we have the following
comparison result (where here recall that $p^*$ denotes the right
adjoint to $Q$ which gives an object of $\gspectra[G]$ the trivial
$\sL_U(1)$-action).

\begin{proposition}
The adjoint pair $(Q,p^*)$ is a weak symmetric monoidal Quillen
equivalence.
\end{proposition}

\begin{proof}
Since $p^*$ preserves fibrations and weak equivalences, this is
clearly a Quillen adjunction.  Taking $\bL \Sg$ as a cofibrant
replacement of the unit in $\gspectra[G]$, we compute that
$Q \bL \Sg \cong \Sg$ and so the adjunction is monoidal.  Finally,
evaluation of $Q$ on the generating cofibrations makes it clear that
the natural map $QX \to UX$ is a weak equivalence for cofibrant $X$,
and so the adjunction is a Quillen equivalence.
\end{proof}

In order to retain homotopical control over $\gspectra_U$, we need to
prove the equivariant analogue of~\cite[I.8.4,XI.2.2]{EKMM}, i.e.,
that the canonical unit map $\lambda \colon \Sg \sma_U X \to X$ is
always a weak equivalence.  The proof of the required result follows
the outline of~\cite[I.8.5]{EKMM}, using Theorem~\ref{thm:mandell}.

\begin{theorem}\label{thm:uniteq}
For any $X$ in $\gspectra[G][\bL_U]$, the unit map
\[
\lambda \colon \Sg \sma_U X \to X
\]
is a weak equivalence.
\end{theorem}

Theorem~\ref{thm:uniteq} now allows us to prove the following theorem.

\begin{theorem}
The category $\gspectra_U$ is a compactly generated symmetric monoidal
proper $G$-topological model category in which the weak equivalences
are detected by the forgetful functor and the fibrations are detected
by the functor $F_U(\Sg,-)$.
\end{theorem}

\begin{proof}
Although $\Sg \sma_U (-)$ is not a monad, the argument
for~\cite[5.13]{mmss} again applies.  As in the corresponding proof
in~\cite[VI.4.6]{EKMM}, consideration of the category of counital
objects in $\gspectra[G][\bL_U]$ is illuminating.
\end{proof}

\begin{remark}
By adjunction, a map $\Sg \sma_U \bL_U S^n \to X$ in $\gspectra_U$ is
the same as a map $\Sg^n \to X$ in $\gspectra$.  As a consequence, the
``internal'' homotopy groups in $\gspectra_U$ determined by the free
objects on spheres coincide with the homotopy groups on the underlying
orthogonal $G$-spectrum.
\end{remark}

The functor $\Sg \sma_U
(-) \colon \gspectra[G][\bL_U] \to \gspectra_U$ is a Quillen left
adjoint and is a symmetric monoidal functor.  In fact, the following
proposition is straightforward to verify.

\begin{proposition}\label{prop:smodquil}
The adjoint pair $(\Sg \sma_U (-), F_U(\Sg,-))$ forms a weak symmetric
monoidal Quillen equivalence between $\gspectra[G][\bL_U]$ and
$\gspectra_U$.
\end{proposition}

As a consequence of these results, we have the following comparison result.

\begin{lemma}
For cofibrant $X,Y \in \gspectra_U$ there is a natural equivalence
\[
X \sma_U Y \to X \sma Y
\]
and more generally for cofibrant $\{X_1, X_2, \ldots
X_n\} \in \gspectra_U$ there are natural equivalences
\[
X_1 \sma_U X_2 \sma_U \ldots \sma_U X_n \to X_1 \sma
X_2 \sma \ldots \sma X_n.
\]
\end{lemma}

We now turn to the study of the multiplicative structure on
$\gspectra_U$.  The following result explains the equivariant
homotopical content of the operadic smash product $\sma_U$.

\begin{theorem}\label{thm:extendedpower}
Let $X$ be a cofibrant object of $\gspectra_U$.  Then there is a
natural weak equivalences of $G \times \Sigma_n$ spectra
\[
(E_{\aF_U}\Sigma_i)_+ \sma X^{\sma i} \htp X^{\sma_U i},
\]
and a natural weak equivalence of $G$-spectra
\[
(E_{\aF_U}\Sigma_i)_+ \sma_{\Sigma_i} X^{\sma i} \htp X^{\sma_U i}
/ \Sigma_i,
\]
where here $\aF_U$ denotes the family of $G \times \Sigma_n$ specified
by $U$.
\end{theorem}

\begin{proof}
When $X$ is free as an object of $\gspectra_U$ (i.e., $X =
S \sma_U \bL_U Y$), then the result follows immediately from
Theorem~\ref{thm:assoc} and the fact that $\sL(U^i, U) \htp
E_{\aF_U} \Sigma_i$.  The general result now follows by inductively
reducing to the free case using the filtration argument
of~\cite[B.117]{HHR}.
\end{proof}

In particular, Theorem~\ref{thm:extendedpower} makes clear the way in
which $\gspectra_U$ depends on the choice of $U$.  Specifically, the
$G \times \Sigma_n$-equivariant homotopy type of the $n$-fold $\sma_U$
power of $X$ is controlled by $U$, and is precisely the universal
space for the family associated to $\sL_U(n)$.

\begin{corollary}\label{cor:preserves}
Let $X \to X'$ be an acyclic cofibration in $\gspectra_U$.  Then the
induced maps
\[
\bT X \to \bT X' \qquad\textrm{and}\qquad \bP X \to \bP X'
\]
are weak equivalences.
\end{corollary}

Corollary~\ref{cor:preserves} provides the essential technical input
for the next theorem, which is again proved using the standard outline
(e.g., see~\cite[5.13]{mmss} or~\cite[B.130]{HHR}).

\begin{theorem}
The categories $\gspectra_U[\bT]$ and $\gspectra_U[\bP]$ are
compactly generated proper $G$-topological model categories with weak
equivalences and fibrations determined by the forgetful functor to
$\gspectra_U$.
\end{theorem}

For a fixed ring object $R$, we have the following relative version of
the preceding theorem.

\begin{theorem}
For an object $R$ in $\gspectra_U[\bT]$ or $\gspectra_U[\bP]$, the
category of $R$-modules in $\gspectra_U$ is a compactly generated
proper $G$-topological model category with weak equivalences and
fibrations determined by the forgetful functor to $\gspectra_U$.  When
$R$ is commutative (i.e., an object in $\gspectra_U[\bP]$), then
\begin{enumerate}
\item the category of $R$-modules in $\gspectra_U$ is a compactly
generated proper $G$-topological symmetric monoidal model category and
\item the category of $R$-algebras is a compactly generated proper
$G$-topological model category.
\end{enumerate}
\end{theorem}

\subsection{The homotopical theory of change of group and fixed-point
functors}

In this section, we describe how to compute the derived functors of
the change-of-group and fixed-point functors described in
Section~\ref{sec:change}.  Our analysis bootstraps from the analogous
theory in the setting of $\gspectra$; the following two lemmas
establish that the homotopical theory for $\gspectra[G][\bL_U]$ and
$\gspectra_U$ can be understood in terms of the homotopical theory for
$\gspectra$.

\begin{lemma}
Let $X$ be an object of $\gspectra[G][\bL_U]$ or $\gspectra_U$.  If
$X$ is cofibrant, then the underlying orthogonal $G$-spectrum
associated to $X$ has the homotopy type of a cofibrant object.  The
analogous results hold for $(\gspectra[G][\bL_U])[\bT]$ and
$\gspectra_U[\bT]$.
\end{lemma}

\begin{proof}
This follows from inspection of the generating cells and
the ``Cofibration Hypothesis'' in this context.  We can assume without
loss of generality that $X$ is a cellular object.  Then $X = \colim_n
X_n$, where the colimit is sequential and along $h$-cofibrations.  The
Cofibration Hypothesis then implies that we can compute the colimit in
the underlying category, and so it suffices to consider each $X_n$.
Since each $X_n$ is formed from $X_{n-1}$ by attaching cells, the
Cofibration Hypothesis again allows us to inductively reduce this to
consideration of the generating cells, where the result is clear.
\end{proof}

\begin{lemma}
Let $X$ be a fibrant object in $\gspectra[G][\bL_U]$,
$\gspectra[G][\bL_U])[\bT]$, or $(\gspectra[G][\bL_U])[\bP]$.  Then
$X$ is fibrant in $\gspectra[G]$.  Analogously, if $X$ is fibrant in
$\gspectra_U$, $\gspectra_U[\bT]$, or $\gspectra_U[\bP]$, then
$F_U(\Sg,X)$ is fibrant in $\gspectra[G]$.
\end{lemma}

\begin{proof}
The statements about modules imply the statements about monoids and
commutative monoids, as fibrations in the model structures on the
categories of algebras are determined by the forgetful functors to
$\gspectra[G][\bL_U]$ and $\gspectra_U$ respectively.  The first
assertion is clear for $\gspectra[G][\bL_U]$ since the fibrations are
created by the forgetful functor to $\gspectra[G]$.  For
$\gspectra_U$, the result follows from
Proposition~\ref{prop:smodquil}; the functor
$F_U(\Sg,-) \colon \gspectra_U \to \gspectra[G][\bL_U]$ is a Quillen
right adjoint.
\end{proof}

In order to understand the behavior of the fixed point functors on
$\gspectra_U$, we need to describe the homotopical behavior of the
point-set multiplicative change of universe functors.  In contrast to
the situation for the additive functors in orthogonal spectra, these
always induce Quillen equivalences.  

\begin{proposition}\label{prop:homochange}
Let $U$ and $U'$ be $G$-universes.  The multiplicative change of
universe functors $\sL\sI_U^{U'}$ are left (and right) Quillen
functors that preserve weak equivalences between cofibrant objects and
therefore induce Quillen equivalences between $\gspectra[G][\bL_U]$
and $\gspectra[G][\bL_{U'}]$ and $\gspectra_U$ and $\gspectra_U'$,
respectively.
\end{proposition}

\begin{warning}
What is not preserved by $\sL\sI_U^{U'}$ is not the underlying
additive homotopy theory but the multiplicative norms.  Specifically,
the derived functor of $\sL\sI_U^{U'}$ preserves only those
multiplicative norms corresponding to $G$-sets that are admissible in
both $U$ and $U'$.  Put another way, these functors do not preserve
the homotopical equivariant symmetric monoidal structure.
\end{warning}

We now turn to the fixed points.  The forgetful functors $\iota_H^*$
preserve all weak equivalences, and so are already derived.  Their
left and right adjoints can be derived by cofibrant or fibrant
approximation, as a consequence of the preceding lemmas.  Similarly,
Proposition~\ref{prop:homochange} implies that the (right) derived
functors of the categorical fixed points can be computed by fibrant
replacement and the (left) derived functors of geometric fixed points
by cofibrant replacement.  We summarize the situation in the following
result.

\begin{proposition}
\hspace{5 pt}
\begin{enumerate}
\item The forgetful functors $\iota_H^*$ preserve all weak equivalences on
$\gspectra_U$ and $\gspectra[G][\bL_U]$.
\item The left adjoint $G_+ \sma_H (-)$ to $\iota_H^*$ preserves weak
equivalences between cofibrant objects on $\gspectra_U$ and
$\gspectra[G][\bL_U]$.  The right adjoint $F_H(G,-)$ to $\iota_H^*$
preserves weak equivalences between fibrant objects on $\gspectra_U$
and $\gspectra[G][\bL_U]$.
\item The categorical fixed point functor $(-)^H$ preserves weak
equivalences between fibrant objects in $\gspectra_U$ and
$\gspectra[G][\bL_U]$.
\item The geometric fixed point functor $\Phi^H$ preserves weak
equivalences between cofibrant objects in $\gspectra_U$ and
$\gspectra[G][\bL_U]$.
\end{enumerate}
\end{proposition}

Finally, we have the following result which shows that the geometric
fixed-point functor is strong monoidal in the homotopical sense.

\begin{proposition}
Let $X$ and $Y$ be cofibrant objects in $\gspectra[G][U]$
or $\gspectra_{U}$.  Then the natural map
\[
\Phi^H X \sma_{U} \Phi^H Y \to \Phi^H
(X \sma_{U} Y)
\]
is a weak equivalence.
\end{proposition}

\begin{proof}
First consider the case of $\gspectra[G][U]$.  The result follows from
the result for $\Phi^H$ on $\gspectra$~\cite[V.4.7]{MM} when $X$ and
$Y$ are generating cells, since
\[
\bL_{U^H} X' \sma_{U^H} \bL_{U^H}
Y' \cong \sL_{U^H}(2)_+ \sma (X' \sma Y')
\]
for any $X'$ and $Y'$ and $WH$ acts trivially on $\sL{U^H}(2)_+$.
Since $(-) \sma_{U^H} (-)$ preserves colimits in either variable and
preserves weak equivalences between cofibrant objects, we can conclude
the general statement.  The case of $\gspectra_U$ follows from
analogous considerations.
\end{proof}

\subsection{The homotopical theory of the norm}\label{sec:homnorm}

In this section, we show that the norm $N_{H,\iota_H^* U}^{G,U}$ is a
homotopical functor and participates in a Quillen adjunction when
restricted to commutative ring objects.

\begin{theorem}\label{thm:normagree}
Let $X$ be a cofibrant object in $\gspectra[G][\bL_U]$ or
$\gspectra_G$.  The natural map 
\[
N_{H,\iota_H^* U}^{G,U} X \to N_H^G X
\]
is a weak equivalence when $G/H$ is admissible for $U$.
\end{theorem}

\begin{proof}
By induction over the cellular filtration, it suffices to consider the
case when $X$ is free.  In this case, we're looking at the map
\[
\sL(\Ind_H^G \iota_H^* U, U)_+ \sma N_H^G X \to N_H^G X
\]
given by the collapse map $\sL(\Ind_H^G \iota_H^* U, U)_+ \to S^0$.
Since the collapse is a $G$-equivalence when $G/H$ is admissible, the
result follows.
\end{proof}

\begin{remark}
When $G/H$ is not admissible for $U$, it is not clear in general what
the homotopy type of $N_{H,\iota_H^* U}^{G,U}$ is.  For free objects,
the homotopy type is controlled by $\sL(\Ind_H^G \iota_H^* U, U)$,
which has no $G$-fixed points.
\end{remark}

\begin{corollary}
The functor $N_{H,\iota_H^* U}^{G,U}$ preserves weak equivalences
between cofibrant objects in $\gspectra[H]_{\iota_H^* U}$ and
$\gspectra[H][\bL_{\iota_H^* U}]$ when $G/H$ is admissible for $U$.
\end{corollary}

The next lemma provides homotopical control on the output of the norm
functor.

\begin{lemma}\label{lem:cof}
Let $X$ be a cofibrant object in $\gspectra[H][\bL_{\widehat{U}}]$ or
$\gspectra[H]_{\widehat{U}}$.  Then $N_{H,\iota_H^* U}^{G,U} X$ is
cofibrant in $\gspectra[G][\bL_U]$.
\end{lemma}

\begin{proof}
Using the filtration of~\cite[A.3.4]{HHR}, we can inductively reduce
to the case when $X$ is of the form $\sL_{\widehat{U}}(1)_+ \sma Y$.  In
this case,
\[
N_{H,\iota_H^* U}^{G,U} (\sL_{\widehat{U}}(1)_+ \sma
Y) \cong \sL(\Ind_{H}^{G} \widehat{U}, U)_+ \sma N_H^G Y.
\]
Since $\sL(\Ind_H^G \widehat{U}, U) \cong \sL_1(U)$ by
Lemma~\ref{lem:unisum}, the result follows.
\end{proof}

As a consequence, we have the following result about the composition
of the norm functor.

\begin{proposition}\label{prop:normcomp}
Fix $H_1 \subseteq H_2 \subseteq G$.  Let $X$ be a cofibrant object in
$\gspectra[H_1][\bL_{\iota_{H_1}^* U}]$ or
$\gspectra[H_1]_{\iota_{H_1}^* U}$.  Then there is a natural weak
equivalence
\[
N_{H_1, \iota_{H_1}^* U}^{G,U} X \htp N_{H_2, \iota_{H_2}^* U}^{G,U}
N_{H_1, \iota_{H_1}^* U}^{H_2, \iota_{H_2}^* U} X.
\]
\end{proposition}

\begin{proof}
Expanding using the definition, we have
\[
N_{H_2, \iota_{H_2}^* U}^{G,U} N_{H_1, \iota_{H_1}^*
 U}^{H_2, \iota_{H_2}^* U} X = \sL(\Ind_{H_2}^G \iota_{H_2}^* U,
 U) \times_{F_{H_2}(G, \sL_{\iota_{H_2}^* U}(1))}
 N_{H_2}^G \left(N_{H_1, \iota_{H_1}^* U}^{H_2, \iota_{H_2}^* U}
 X\right)
\]
and
\[
N_{H_1, \iota_{H_1}^* U}^{H_2, \iota_{H_2}^* U}
 X = \left(\sL(\Ind_{H_1}^{H_2} \iota_{H_1}^* U, \iota_{H_2}^*
 U) \times_{F_{H_1}(H_2, \sL_{\iota_{H_1}^* U}(1))} N_{H_1}^{H_2}
 X\right)
\]
which implies that
\[
N_{H_2}^G \left(N_{H_1, \iota_{H_1}^* U}^{H_2, \iota_{H_2}^* U}
 X\right)
\cong
\left(F_{H_2}(G,\sL(\Ind_{H_1}^{H_2} \iota_{H_1}^* U, \iota_{H_2}^* U)) \times_{F_{H_2}(G, \sL_{\iota_{H_1}^* U}(1))} N_{H_1}^{G} X\right).
\]
Next, we show that
\[
\sL(\Ind_{H_2}^G \iota_{H_2}^* U, U) \times_{F_{H_2}(G, \sL_{\iota_{H_2}^* U}(1))} F_{H_2}(G,\sL(\Ind_{H_1}^{H_2} \iota_{H_1}^* U, \iota_{H_2}^* U))
\]
is isomorphic to $\sL(\Ind_{H_1}^G \iota_{H_1}^* U, U)$.  There is an equivariant map
\[
\sL(\Ind_{H_2}^G \iota_{H_2}^* U, U) \times F_{H_2}(G,\sL(\Ind_{H_1}^{H_2} \iota_{H_1}^* U, \iota_{H_2}^* U)) \to \sL(\Ind_{H_1}^G \iota_{H_1}^* U, U)
\]
induced by composition and the natural map
\[
F_{H_2}(G,\sL(\Ind_{H_1}^{H_2} \iota_{H_1}^* U, \iota_{H_2}^* U)) \to \sL(\Ind_{H_1}^{G} \iota_{H_1}^* U, \Ind_{H_2}^G \iota_{H_2}^* U)
\]
induced by the direct sum.  This map is compatible with the maps
determining the coequalizer, and so it suffices to check that the
underlying non-equivariant diagram is a reflexive coequalizer.  This
now follows from Lemma~\ref{lem:lemhop2}.  The theorem is now a
consequence of the preceding isomorphism and Lemma~\ref{lem:cof}.
\end{proof}

In the case of commutative monoid objects, it is straightforward to
check that the adjunction involving the norm and the forgetful functor
is homotopical; it is clear that $\iota_H^*$ preserves fibrations and
weak equivalences.

\begin{theorem}\label{thm:homoadj}
The adjoint pairs
\[
N_{H,\iota_H^* U}^{G,U} \colon (\gspectra[H][\bL_{\iota_H^*
U}])[\bP] \leftrightarrows
(\gspectra[G][\bL_{U}])[\bP] \colon \iota_H^*
\]
and
\[
N_{H,\iota_H^* U}^{G,U} \colon \gspectra[H]_{\iota_H^*
U}[\bP] \leftrightarrows
\gspectra_U[\bP] \colon \iota_H^*
\]
are Quillen adjunction.
\end{theorem}

Note however that the derived functor of the norm $N_{H,\iota_H^*
U}^{G,U}$ on commutative rings only agrees with the derived functor of
the module norm when $G/H$ is admissible for $U$; the following result
is a consequence of the fact that derived functor of the norm on
commutative rings in orthogonal $H$-spectra agrees with the underlying
norm~\cite[B.148]{HHR}.

\begin{proposition}
Let $X$ be a cofibrant object in $\gspectra[G][\bL_U][\bP]$.  The
natural map
\[
N_{H,\iota_H^* U}^{G,U} X \to N_H^G X
\]
is a weak equivalence when $G/H$ is admissible for $U$.
\end{proposition}

We now turn to the relative norm construction.

\begin{theorem}\label{thm:rnormagree}
The functor $_R N_{H, \iota_H^* U}^{G,U}$ preserves weak equivalence
between cofibrant objects in $\aM_{\iota_H^* R,\iota_H^* U}$ when
$G/H$ is admissible for $U$.
\end{theorem}

\begin{proof}
Since the $R$-relative norm is strong symmetric monoidal, it suffices
to show that when $X$ is cofibrant in $\aM_{\iota_H^* R,\iota_H^* U}$,
$N_{H, \iota_H^* U}^{G,U} X$ is cofibrant as an $N_{H,\iota_H^*
U}^{G,U} R$ module.  Once again, it suffices to check this on free
objects, where it is straighforward.
\end{proof}

\section{$G$-symmetric monoidal categories of modules over an
$N_\infty$ algebra}\label{sec:Gsym}

In this section, we describe the homotopical $G$-symmetric monoidal
structure on $\aM_{R,U}$.  More precisely, we have a $\sL_U$-symmetric
monoidal structure, where we mean an equivariant symmetric monoidal
structure specified by the coefficient system of admissible sets for
$\sL_U$~\cite[4.4]{HillHopkinsLocalization}.  We characterize this
structure in terms of a homotopical exponential functor 
\[
N^T \colon \aM_{R,U} \to \aM_{R,U}
\]
for any admissible $G$-set $T$.  We explain how this ``internal norm'' arises
from structure on the collection of norms and forgetful functors on
the categories $\aM_{\iota_H^* R, \iota_H^* U}$ as $H$ varies over the
closed subgroups of $G$; these functors assemble into an incomplete
Mackey functor in homotopical categories.  We also explain the
resulting structure on commutative monoid objects, recovering the
characterizations of~\cite[6.11]{BlumbergHill}.

\subsection{The $G$-symmetric monoidal structure on $\gspectra$ and $\aM_R$}\label{sec:gsym}

In this subsection, we review the canonical $G$-symmetric monoidal
structure on $\gspectra$ and $\aM_R$ for $R$ a commutative ring
orthogonal $G$-spectrum.  We begin by recalling
from~\cite[\S6]{BlumbergHill} the definition of the internal norm in
orthogonal spectra.

\begin{definition}
Let $H \subset G$ be a closed subgroup.  The internal norm of an
orthogonal $G$-spectrum $X$ is specified by the formula
\[
N^{G/H} X = N_H^G \iota_H^* X.
\]
For an arbitrary $G$-set $T$, we define the internal norm by
decomposing $T$ into a disjoint union of orbits $\coprod_i G/H_i$ and
defining
\[
N^T X = \bigwedge_i N^{G/H_i} X.
\]
\end{definition}

For example, when $T$ is a trivial $G$-set, $N^T M$ is simply the
smash-power of $|T|$ copies of $M$.  Note that this definition extends
in the evident way to categories of modules over a commutative ring
orthogonal $G$-spectrum.

There is another equivalent description for this which will make the properties of the norm (summarized in Theorem~\ref{thm:gsymmabs} below) more transparent. If $T$ is a finite $G$-set, then let $B_TG$ denote the translation category of $T$. This has object set $T$ itself and the morphism set is $T\times G$ with structure maps the projection onto $T$ and the action. Given a $G$-spectrum $X$, we have a $B_GT$-shaped diagram $X^{T}$ described by $t\mapsto X$ and $(t,g)$ acts as multiplication by $g$ on $X$.

A map of finite $G$-sets $f\colon T\to S$ produces a covering category $B_TG\to B_SG$ as in \cite[Definition A.24]{HHR}, and therefore we have an associated indexed monoidal product $f_\ast^\otimes$.

\begin{proposition}\label{prop:RemixNorm}
Let $p\colon T\to\ast$ be the terminal map. There is a canonical isomorphism
\[
p_\ast^\otimes X^{T}\cong N^T(X).
\]
\end{proposition}

\begin{proof}
Since both sides take disjoint unions to smash products (the left by
construction and the right by definition), it suffices to construct
the canonical isomorphism when $T=G/H$.  In this case, each side is
then an indexed product.

Using the additive change of universe equivalence, we can work in the
category of orthogonal spectra with a $G$-action and prove the desired 
equality there. In this case, both sides are the indexed product
associated to $p\colon G/H\to \ast$, so it will suffice to show that
the resulting diagrams are isomorphic. For the left-hand side, the
diagram is the constant diagram $X^{G/H}$. For the right-hand side,
the diagram is determined by choosing coset representatives and
sending a coset $gH$ to the $gHg^{-1}$-spectrum $g\cdot i_H^\ast X$
(where here $g\cdot Y$ for an $H$-spectrum $Y$ is just the restriction
along the isomorphism $gHg^{-1}\cong H$). However, we then have an
equivariant isomorphism of diagrams 
\[
X^{G/H}\cong (gH\mapsto g\cdot i_H^\ast X)
\]
which at a coset $gH$ is simply multiplication by $g$. The indexed products are therefore isomorphic.
\end{proof}

\begin{remark}
The key step in the argument is the same as the one showing that we have canonical isomorphisms
\[
F_H(G_+,i_H^\ast X)\cong F(G/H_+,X)\text{ and } G_+\sma_H i_H^\ast X\cong G/H_+\sma X.
\]
In each case, we have the same two diagrams as the one given above and then we compare the associated indexed monoidal products.
\end{remark}

Because $N_H^G$, $-\sma -$, and $\iota_H^*$ preserve weak equivalences and
cofibrant objects, the internal norm is a homotopical functor.

\begin{lemma}\label{lem:inthom}
For any $G$-set $T$, the internal norm $N^T$ preserves acyclic
cofibrations.
\end{lemma}


We can now recall the basic theorem establishing the $G$-symmetric monoidal structure on $G$-spectra. All of this follows easily from Appendix A of \cite{HHR}; for convenience, we include details here.

\begin{theorem}\label{thm:gsymmabs}
\mbox{}
\begin{enumerate}
\item For $H_1 \subseteq H_2 \subseteq G$, there is a natural
isomorphism
\[
\iota_{H_1}^* \cong \iota_{H_1}^* \iota_{H_2}^*.
\]
\item For $G$-sets $T_1$ and $T_2$,
there is a natural isomorphism
\[
N^{T_1 \times T_2} X \htp N^{T_1} N^{T_2} X.
\]
\item For $K \subset H$, there is a natural isomorphism
\[
\iota_K^* N^T X \htp N^{\iota_K^* T} \iota_K^* X
\]
\end{enumerate}
\end{theorem}
\begin{proof}
The first part is obvious.  For the second and third parts, we use the alternative description of $N^TX$ given by Proposition~\ref{prop:RemixNorm}.

For the second, observe that $B_{T_1\times T_2} G\cong B_{T_1}G\times B_{T_2}G$, and the composite of the norms is the composites of the indexed products
\[
T_1\times T_2\to T_1\to \ast.
\]
The composite of the indexed products is the indexed product of the composites \cite[Prop A.29]{HHR}.

The third is the variant of the double coset formula here. If $T$ is a finite $G$-set, then we have a pullback diagram of categories
\[
\xymatrix{
{B_{G/H\times T}G}\ar[r]\ar[d] & {B_T G}\ar[d] \\
{B_{G/H} G}\ar[r] & {BG.}
}
\]
Since $G/H\times T\cong G\times_H i_H^\ast T$, the left-hand side of this diagram is equivalent to $B_{i_H^\ast T} H\to BH$. Since the map on spectra induced by pulling back along $B_{G/H} G\to BG$ is $i_H^\ast$, we conclude by \cite[Prop A.31]{HHR} that $i_H^\ast N^TX\cong N^{i_H^\ast T}i_H^{\ast X}$.
\end{proof}

The analogue of Theorem~\ref{thm:gsymmabs} for modules over a
commutative ring orthogonal $G$-spectrum $R$ follows from the
characterization of the $R$-relative norm via the formula
\[
_R N_H^G X \cong R \sma_{N_H^G R} N_H^G X
\]
and the fact that the norm $N_H^G$ is the left adjoint to the
restriction functor $\iota_H^*$ on commutative rings.  We explain in
detail the argument below in the proof of Theorem~\ref{thm:gsymmmon}.

\subsection{The $\sL_U$-symmetric monoidal structure on $\gspectra_U$ and $\aM_{R,U}$}

We now provide the analogous definitions in our context.

\begin{definition}
Given $H \subset G$ a closed subset, we define the internal norm
\[
_R N_U^{G/H} M \colon \aM_{R,U} \to \aM_{R,U}
\]
as the composite
\[
_R N^{G/H}_U (-) := _R N^{G,U}_{H,\iota_H^*} \iota_H^* (-).
\]
We extend the internal norm to an arbitrary $G$-set $T$ by decomposing
$T$ into a disjoint union of orbits $\coprod_i G/H_i$ and specifying
that
\[
_R N^T_U M = \bigwedge_{i} {_R N^{G/H_i}} M.
\]
\end{definition}

We now describe the homotopical properties of the internal norm.  We
begin by considering the absolute case where $R = S$.

\begin{lemma}\label{lem:nderive}
Let $T$ be an admissible $G$-set.  The functor $N^T$ preserves weak
equivalences between cofibrant objects.
\end{lemma}

\begin{proof}
This is a consequence of the fact that $\iota_H^*$ preserves cofibrant
orthogonal $G$-spectra~\cite[V.2.2]{MM}, colimits, and the
identification
\[
\iota_H^* \sL_U(1)_+ \sma X \cong \sL_{\iota_H^*U}(1)_+ \sma
(\iota_H^* X).
\]
\end{proof}

Furthermore, we can also identify the interaction of $N^T$ with the
cartesian product.

\begin{lemma}\label{lem:normcart}
When $T_1$ and $T_2$ are admissible $G$-sets and $M$ is a cofibrant
object in $\gspectra[G][\bL_U]$, there is a natural weak equivalence
\[
N^{T_1 \times T_2} M \htp N^{T_1} (N^{T_2} M).
\]
\end{lemma}

\begin{proof}
This follows from Lemma~\ref{lem:nderive}, Lemma~\ref{lem:inthom}, and
Theorem~\ref{thm:normagree}.
\end{proof}

Proposition~\ref{prop:normcomp} shows that the norm functors compose
as expected, and it is clear that for $H_1 \subseteq H_2 \subseteq G$,
$\iota_{H_1}^* \cong \iota_{H_1}^* \iota_{H_2}^*$.

\begin{theorem}\label{thm:doublecoset}
Fix $K \subseteq H$, let $T$ be an admissible $H$-set, and let $M$ be
a cofibrant object in $\gspectra[G][\bL_U]$.  The composite $\iota_K^*
N^T M$ is naturally equivalent to $N^{\iota_K^* T} \iota_K^* M$.
\end{theorem}

\begin{proof}
This again follows from Theorem~\ref{thm:normagree} and the fact that
the desired equivalence holds for the norm in orthogonal spectra.
\end{proof}

When $R$ is no longer necessarily the sphere, we have corresponding
analogues of the preceding results; we summarize the situation in the
following theorem.

\begin{theorem}\label{thm:gsymmmon}
Let $R$ be a cofibrant object in $\gspectra_U[\bP]$.
\begin{enumerate}
\item The functor ${_R} N^T_U$ preserves weak equivalences between
cofibrant objects.
\item When $T_1$ and $T_2$ are admissible $G$-sets and $M$ is a
cofibrant object in $\aM_{R,U}$, there is a natural equivalence
\[
_R N_U^{T_1 \times T_2} M \htp _R N_U^{T_1} (_R N_U^{T_2} M).
\]
\item For $H_1 \subseteq H_2 \subseteq G$, there is a natural
isomorphism
\[
\iota_{H_1}^* \cong \iota_{H_1}^* \iota_{H_2}^*.
\]
\item For $K \subset H$ and $T$ an admissible $G$-set, there is a
natural equivalence
\[
\iota_K^* {_R} N_{U}^T M \htp {_R} N_{U}^{\iota_K^* T} \iota_K^* M
\]
when $M$ is a cofibrant object in $\aM_{R,U}$.
\end{enumerate}
\end{theorem}

\begin{proof}
The first of these follows from Theorem~\ref{thm:rnormagree}.  The
second is a consequence of Theorem~\ref{thm:normagree}; the proof is
analogous to the proof of Lemma~\ref{lem:normcart}, along with the
observation that the smash product defining the relative norm computes
the derived smash product under our hypotheses.  The third is
immediate.  For the fourth, we can leverage the absolute result as
follows.

Since $\iota_K^*$ is a strong symmetric monoidal functor, we have the
isomorphisms
\[
\iota_K^* {_R} N_{U}^T M \cong \iota_K^* \left( N_U^T M \sma_{N_U^T R}
R \right) \cong (\iota_K^* N_U^T M) \sma_{\iota_K^* N_U^T R} \iota_K^* R.
\]
By Theorem~\ref{thm:doublecoset}, we know that
\[
\iota^*_K N_U^T M \htp N^{\iota_K^* T} \iota_K^* M.
\]
Moreover, since $N_U^T$ is a left adjoint on commutative rings, we
have an isomorphism
\[
\iota_K^* N_U^T R \cong N_{\iota^*_H U}^{\iota_H^* T} \iota_H^* R,
\]
which is compatible with the counit $N_U^T R \to R$ used in the
formation of the relative smash product.  Since the hypotheses
guarantee we are computing the derived smash product, we end up with a
natural weak equivalence
\[
\iota_K^* {_R} N_{U}^T M \htp N^{\iota_K^* T} \iota_K^*
M \sma_{N_{\iota^*_H U}^{\iota_H^* T} \iota_H^* R} \iota_H^* R \cong {_R} N_{U}^{\iota_K^* T} \iota_K^* M.
\]
\end{proof}

\subsection{The multiplicative structure on $N_\infty$ algebras}

In this subsection, we explain how the $\sL_U$-symmetric monoidal
structure on $\gspectra_U$ induces additional multiplicative structure
on objects of $\gspectra_U[\bP]$.  Of course,
Theorem~\ref{thm:commninf} implies that an object of $\gspectra_U[\bP]$
is an $N_\infty$ algebra structured by the equivariant linear
isometries operad determined by $U$, and~\cite[6.11]{BlumbergHill}
explains the extra structure this gives.  Our purpose here is to
demonstrate that this structure is essentially an immediate
consequence of Theorem~\ref{thm:gsymmmon}.

Let $R$ be a cofibrant object of $\gspectra_U[\bP]$.  The adjunction
of Theorem~\ref{thm:homoadj} yields homotopical counit maps
\[
N^{G/H}_U = N_{H, \iota_H^*}^G \iota_H^* R \to R
\]
for admissible $G/H$, which clearly induce natural maps
\[
N_U^T R \to R \qquad\textrm{and}\qquad G_+ \sma_H N_U^S \iota_K^*
R \to R,
\]
for admissible $G$-sets $T$ and admissible $K \subseteq G$ sets $S$.
The argument of~\cite[6.8]{BlumbergHill} extends without change to
produce a map
\[
N_U^T R \to N_U^S R
\]
given any $G$-map $f \colon S \to T$.

It is clear from the definition of $N_U^T$ that the diagram
\[
\xymatrix{
N_U^{S \coprod T} R \cong N^S_U R \sma N^T_U R \ar[r] \ar[d] & R \sma
R \ar[dl] \\
R \\
}
\]
commutes.  Assertion (2) of Theorem~\ref{thm:gsymmmon} implies that
the diagram
\[
\xymatrix{
N_U^{S \times T} R \cong N_U^S N_U^T R \ar[d] \ar[r] & N_U^T
R \ar[dl] \\
R \\
}
\]
commutes.  Finally, assertion (4) of Theorem~\ref{thm:gsymmmon}
implies that for any admissible sets $S$ and $T$ such that for some
$K \subseteq G$ we have $\iota_K^* S \cong \iota_K^* T$, the diagram
\[
\xymatrix{
\iota_K^* N_U^S R \cong N_{\iota_K^* U}^{\iota_K^* S} \iota_K^*
R \ar[rr]^-{\cong} \ar[dr] && N_{\iota_K^* U}^{\iota_K^* T} \iota_K^*
R \cong \iota_K^* N_U^T R \ar[dl] \\
& R &
}
\]
commutes.  Thus, we precisely recover the characterizations
of~\cite[6.11]{BlumbergHill}.

\section{Examples and applications}

We close with several examples in which the technology in this paper
can be used to construct symmetric monoidal structures on categories
of equivariant modules.  All of the examples we consider arise when
studying smashing Bousfield localization.  The first family of
examples is a necessary ingredient in the work of Greenlees and
Shipley on monoidal equivalences between various models for rational
$G$-spectra~\cite{GreenleesShipley, GreenleesShipley2}.  The second
class of examples is relevant to understanding chromatic localizations
in the equivariant setting.

The technical underpinning of all of these results is the following
theorem of Hopkins and the second
author~\cite{HillHopkinsLocalization}.

\begin{theorem}\label{thm:loc}
Let $\cO$ be an $\Ninfty$ operad, and let $\m{\mathcal C}_{\cO}$
denote the associated indexing system. Let $L$ be a Bousfield
localization on the category $\gspectra$ and let $\m{\mathcal Z}$
denote the coefficient system of acyclics for $L$ (i.e. the value at
$G/H$ is the subcategory of the homotopy category of $\gspectra[H]$
consisting of those $H$-spectra which are acyclic for the restriction
of $L$). Then if $\m{\mathcal Z}$ is closed under the (derived) norms
specified by $\m{\mathcal C}_{\cO}$, $L$ preserves $\cO$-algebras.
\end{theorem}

In particular, this theorem reduces questions about what structure a
localization preserves to determining categorical structure on the
categories of acyclics.  We apply this in two cases of classical
interest.

\subsection{Isotropic localization}

As was first observed by McClure~\cite{McClure}, the localization
which nullifies anything induced does not preserve genuine equivariant
commutative rings (e.g., algebras over the linear isometries operad
for a complete universe $U$).  In particular,
$\Sigma^{\infty} \tilde{E}\cP$ cannot be made into a genuine
equivariant commutative ring spectrum: since the restriction to any
proper subgroup of $\Sigma^{\infty}\tilde{E}\cP$ is contractible, then
the putative counit map determined by the commutative ring structure
\[
N_H^Gi_H^\ast \Sigma^{\infty}\tilde{E}\cP\to\Sigma^{\infty}\tilde{E}\cP
\]
cannot be unital.

More generally, we can apply Theorem~\ref{thm:loc} to produce
immediate strengthenings of this observation.  Let $\cF$ be a family
of subgroups of $G$.  For any $\cF$ there exists a smashing
localization $L_{\cF}$ which nullifies any $G$-spectrum with isotropy
in $\cF$.  The canonical localization sequence is then precisely the
isotropy separation sequence:
\[
E\cF_{+}\wedge X\to X\to \tilde{E}\cF\wedge X.
\]

\begin{proposition}\label{prop:families}
Let $\cF$ be a family of subgroups of $G$ which is not the trivial
family.  Then $\tilde{E}\cF$ is not a genuine equivariant
commutative ring spectrum.  (It is however always a naive $E_\infty$
ring spectrum.)
\end{proposition}

\begin{proof}
If $\cF$ is non-trivial, then $i_e^\ast \Sigma^{\infty} \tilde{E}\aF$ is
contractible. The argument above now shows that if
$\Sigma^{\infty} \tilde{E}\aF$ had a genuine equivariant commutative
ring structure, then the absolute norm would factor through the zero
ring; we arrive at the same contradiction as above.  The final
observation is always satisfied by Bousfield localizations. 
\end{proof}

Proposition~\ref{prop:families} shows that category of local objects
(equivalently, the category of modules over
$\Sigma^{\infty} \tilde{E}\mathcal F$) cannot be given a symmetric
monoidal structure when working with the symmetric monoidal category
of orthogonal $G$-spectra.  In contrast, using Theorem~\ref{thm:main1} 
above, we can obtain a symmetric monoidal category of modules.

\begin{corollary}
For any family $\cF$ of subgroups of $G$, the category of local
spectra for $L_{\cF}$ can always be modeled by a symmetric monoidal
category.  
\end{corollary}

More interestingly, we can describe localizations that result in
richer equivariant structures on categories of local objects (i.e.,
modules).

\begin{theorem}
Let $\cF$ be a family of subgroups of $G$. Let $\sL_U$ be such that
for all admissible $H/K$ and for $H'$ in the family, the isotropy of
\[
N_K^Hi_K^\ast (\Sigma^{\infty}_+ H/H')=\Sigma^{\infty}_{+} Map_K(H,H/H')
\]
is in $\mathcal F$.  Then $\Sigma^{\infty}\tilde{E}\mathcal F$ is a
$\sL_U$-algebra and its category of modules is a $\sL_U$-symmetric
monoidal category.
\end{theorem}

This provides a very satisfying sanity check. If $N$ is a normal
subgroup of $G$ and if $\cF_N$ is the family of subgroups which do not
contain $N$, then there is a composite Quillen equivalence 
\[
\modules{\Sigma^{\infty}\tilde{E}\mathcal F_N} \leftrightarrows \gspectra[(G/N)],
\]
where the right adjoint is essentially just the $N$-fixed points
(e.g., see~\cite[3.2,3.3]{GreenleesShipley2}).  The target is a
$\m{\Set}^{G/N}$-monoidal category as recalled in Section~\ref{sec:gsym}
above.  Our work can be used to promote this Quillen equivalence to a
structured equivalence via the following result.

\begin{corollary}
Let $N$ be a normal subgroup of $G$, and let $\cF_{N}$ denote
the family of subgroups of $G$ which do not contain $N$.  Then the
category of $\Sigma^{\infty}\tilde{E}\mathcal F_N$-modules can be
modeled as a $\m{\Set}^{G/N}$-symmetric monoidal category.
\end{corollary}

\subsection{Chromatic localization}

Work of Balmer and Sanders has (up to a small ambiguity) classified
the triangulated subcategories of $\gspectra$~\cite{BalmerSanders}.
These are determined by the topology on the spectrum (in the sense of
Balmer~\cite{Balmer}) of $\gspectra$: triangulated subcategories of
$\gspectra$ are in bijective correspondence with Thomason subsets of
the spectrum, i.e., the subsets which are a union of closed subsets
with quasi-compact complement.  Balmer and Sanders showed that the
prime ideals are exactly the inverse images under various geometric
fixed points functors of the classical Devinatz-Hopkins-Smith type
$n$-spectra.

Given a Thomason subset $V$, let $L_V$ denote the associated
localization nullifying the triangulated subcategory associated to
$V$.  Theorem~\ref{thm:loc} above specifies when $L_V$ preserves
equivariant multiplicative structures (and a complete classification
of such localizations is forthcoming), so we single out a particular
case of interest.

Fix a prime $p$ such that $p | |G|$ and let $(\gspectra)_p$ denote the
category $\gspectra$ localized at $p$.  Let $V_{n,G}$ denote the
triangulated subcategory of $(\gspectra)_p$ generated by $G_+\wedge
M(n)$, where $M(n)$ is any type $n$-spectrum.

\begin{proposition}
The localization $L_{V_{n,G}}$ does not preserve genuine equivariant
commutative ring spectra.
\end{proposition}

\begin{proof}
Everything in the triangulated category $V_{n,G}$ has the property
that the geometric fixed points are contractible.  However, the
diagonal map provides an isomorphism in the derived category
\[
E \cong \Phi^G N_e^G E
\]
for any spectrum $E$.  In particular, taking 
\[
E=i_e^\ast G_+\wedge M(n)\simeq \bigvee_{|G|} M(n)
\]
shows that the geometric fixed points of the norm of the generator of
the acyclics is not acyclic.
\end{proof}

In particular, there is little hope for any of the equivariant
chromatic categories to be $G$-symmetric monoidal categories.  Once
again, Theorem~\ref{thm:main1} above guarantees that we can construct
models that are symmetric monoidal categories, however.

\begin{appendix}

\section{The equivariant linear isometries operad}\label{app:lin}

In this section, we collect some technical results about the behavior
of the equivariant linear isometries operad.  

\begin{lemma}\label{lem:unisum}
Let $U$ be any $G$-universe.  If $T$ is a non-empty admissible set for
$\sL(U)$, then there is a $G$-equivariant homeomorphism
\[
\mathbb R\{T\}\otimes U\to U.
\]
\end{lemma}

\begin{proof}
By definition of admissibility, for the linear isometries operad we
have an equivariant embedding 
\[
\mathbb R\{T\}\otimes  U\to U.
\]
This implies that every isomorphism class of representations in $\mathbb R\{T\}\otimes U$ is contained in $U$. The inclusion of a trivial summand in $\mathbb R\{T\}$ (which exists since $T$ is non-empty) guarantees that every irreducible representation of $U$ is also in $\mathbb R\{T\}\otimes U$.
\end{proof}

\begin{lemma}\label{lem:equikriz}
The orbit space $\sL_U(2) / (\sL_U(1) \times \sL_U(1))$ consists of a
single point.  More generally, the orbit space $\sL_U(n)
/ \sL_U(1)^{\times n}$ consists of a single point.
\end{lemma}

\begin{proof}
The right action map
$\sL_U(2) \times \sL_U(1) \times \sL_U(1) \to \sL_U(2)$ is clearly a
map of $G$-spaces.  As a consequence, we can compute the orbit space
as the colimit of underlying spaces, and so in this case the result
follows from the non-equivariant identification of the orbit
space~\cite[I.8.1]{EKMM}.

We deduce the general case by induction: We can use
theorem~\ref{thm:assoc} to write
\[
\sL_U(n) / \sL_U(1)^{\times n} \cong
\left(\sL_U(2) \times_{\sL_U(1) \times \sL_U(1)}
(\sL_U(1) \times \sL_U(n-1))\right) / \sL_U(1)^{\times n}.
\]
Since coequalizers commute, the result for $n$ now follows from the
base case $n=2$ and the induction hypothesis.
\end{proof}

More generally, we have the following result.

\begin{lemma}\label{lem:normkriz}
The orbit space $\sL(\Ind_H^G \widehat{U}, U) /
 F_H(G, \sL_{\widehat{U}}(1))$ consists of a single point.
\end{lemma}

\begin{proof}
As in the proof of Lemma~\ref{lem:equikriz}, since the action map
\[
\sL(\Ind_H^G \widehat{U}, U) \times F_H(G, \sL_{\widehat{U}}(1)) \to \sL(\Ind_H^G \widehat{U}, U)
\]
is a map of $G$-spaces, it suffices to compute the orbit space in
terms of the colimit of the underlying spaces.  In this case, we can
deduce the result from Lemma~\ref{lem:equikriz}.
\end{proof}

We also have a series of generalizations of~\cite[I.5.4]{EKMM}.

\begin{lemma}\label{lem:normhop}
Let $T$ and $T'$ be non-empty admissible sets for $U$.  There are
natural isomorphisms
\begin{align*}
\sL(\bR\{T\} &\otimes \widehat{U} \oplus \bR\{T'\} \otimes \widehat{U},
U) \cong \\
&\sL_U(2) \times_{\sL_U(1) \times \sL_U(1)}
(\sL(\bR\{T\} \otimes \widehat{U}, U) \times \sL(\bR\{T'\} \otimes \widehat{U}, U)).
\end{align*}
\end{lemma}

\begin{proof}
First, observe that it suffices to show that non-equivariantly this
isomorphism arises from a reflexive coequalizer diagram.  Now using
Lemma~\ref{lem:unisum} to choose isomorphisms
$\widehat{U} \otimes \bR\{T'\} \cong \widehat{U}$, the required
non-equivariant splittings arise just as in the proof
of~\cite[I.5.4]{EKMM}.
\end{proof}

A particularly useful corollary of Lemma~\ref{lem:normhop} is the following:

\begin{corollary}\label{cor:corhop}
There is a natural isomorphism
\[
\sL(\Ind_H^G \widehat{U} \oplus \Ind_H^G \widehat{U},
U) \cong \sL_U(2) \times_{\sL_U(1) \times \sL_U(1)}
(\sL(\Ind_H^G \widehat{U}, U) \times \sL(\Ind_H^G \widehat{U}, U)).
\]
\end{corollary}

We also have another kind of analogue of~\cite[I.5.4]{EKMM}.

\begin{lemma}\label{lem:lemhop2}
Let $T$ be a non-empty admissible set for $U$.  Then there is a
natural isomorphism
\[
\sL((\bR\{T\} \otimes \widehat{U}) \oplus (\bR\{T\} \otimes \widehat{U}), U) \cong \sL(\bR\{T\} \otimes \widehat{U}, U) \times_{\sL_{\widehat{U}}(1)^{T}} \sL_{\widehat{U}}(2)^{T}.
\]
\end{lemma}

\begin{proof}
Again, the result follows by producing a reflexive coequalizer after
forgetting the $G$-action.  Specifically, we need to show that the
diagram
\[
\xymatrix{
\sL(\bR\{T\} \otimes \widehat{U},
U) \times \sL_{\widehat{U}}(1)^{T} \times \sL_{\widehat{U}}(2)^{T}
\ar@<1ex>[d] \ar@<-1ex>[d] \\
\sL(\bR\{T\} \otimes \widehat{U},
U) \times \sL_{\widehat{U}}(2)^{T} \ar[d] \\
\sL((\bR\{T\} \otimes \widehat{U}) \oplus (\bR\{T\} \otimes \widehat{U}), U)
}
\]
is a reflexive coequalizer.  Choosing $|T|$ isomorphisms
$h_i \colon \widehat{U}^2 \cong \widehat{U}$ such that the sum assembles to
an isomorphism
$h \colon \bR\{T\} \otimes \widehat{U} \oplus \bR\{T\} \otimes \widehat{U} \cong \bR\{T\} \otimes \widehat{U}$,
we can define the splitting map
\[
\sL((\bR\{T\} \otimes \widehat{U}) \oplus
(\bR\{T\} \otimes \widehat{U}), U) \to
\sL(\bR\{T\} \otimes \widehat{U}, U) \times \sL_{\widehat{U}}(2)^{T}
\]
via $f \mapsto (f \circ h, h_1, h_2, \ldots, h_{|T|})$.  The argument
now proceeds exactly as in~\cite[I.5.4]{EKMM}.
\end{proof}

This has the following corollary.

\begin{corollary}\label{cor:corhop2}
For $H \subseteq G$, there is a natural isomorphism
\[
\sL(\Ind_H^G \widehat{U} \oplus \Ind_H^G \widehat{U},
U) \cong \sL(\Ind_H^G \widehat{U}, U) \times_{F_H(G, \sL_{\widehat{U}}(1))}
F_H(G, \sL_{\widehat{U}}(2)).
\]
\end{corollary}

Finally, we turn to the main technical theorem about the equivariant
linear isometries operad that justifies the use of the unital objects.
In the proof, we make use of the following standard technical lemma:

\begin{lemma}\label{lem:action}
Let $X$ have a left $H$-action and right $G$-action which are
compatible (i.e., $X$ is an $H \times G$-space).  Then the coequalizer
\[
(-) \times_{H} X
\]
specifies a functor from the category of $G' \times H$-spaces and
equivariant maps to $G' \times G$-spaces and equivariant maps.
\end{lemma}

\begin{proof}
Let $Y$ be a $G' \times H$-space.  It is clear that $Y \times_H X$ has
a $G' \times G$ action inherited from the $G'$-action on $Y$ and the
$G$-action on $X$.  Let $f \colon Y \to Y'$ be a map of $G' \times
H$-spaces.  Then there is an induced map of spaces
\[
\theta_f \colon Y \times_H X \to Y' \times_H X
\]
defined by $(y,x) \mapsto (f(y),x)$.  This is a left $G'$-map since $f$ is
a $G' \times H$-map; $\theta_f((g'y,x)) = (f(g'y),x) = (g'f(y),x) =
g'\theta_f((y,x))$.  Similarly, it is a right $G$-map.
\end{proof}

\begin{theorem}\label{thm:mandell}
For each $k > 0$, the map
\[
\gamma_k \colon \hat{\sL}_U(k) = \sL_U(2) \times_{\sL_U(1) \times \sL_U(1)}
(\sL_U(0) \times \sL_U(k)) \to \sL_U(k)
\]
induced by the operadic structure map
\[
\sL_U(2) \times \sL_U(0) \times \sL_U(k) \to \sL_U(k)
\]
is a homotopy equivalence of $G \times \Sigma_k$ spaces.
\end{theorem}

\begin{proof}
First, consider the case where $k=1$.  In this case, we are considering the map
\[
\gamma_1 \colon \sL_U(2) \times_{\sL_U(1) \times \sL_U(1)} (\sL_U(0) \times \sL_U(1)) \to \sL_U(1).
\]
The proof of~\cite[XI.2.2]{EKMM} goes through in the equivariant
context to show that $\gamma_1$ is a homotopy equivalence of
$G$-spaces.  It is helpful to decompose $\gamma_1$ as
follows~\cite[VI.6]{MM}:
\[
\xymatrix{
\sL_U(2) \times_{\sL_U(1) \times \sL_U(1)}
(\sL_U(0) \times \sL_U(1)) \ar[r]^-{\theta_1} & \sL_U(2)
/ \sL_U(1) \ar[r]^-{\theta_2} & \sL_U(1),\\
}
\]
where $\sL_U(2)/\sL_U(1)$ is the orbit space for the right action of
$\sL_U(1)$ on $\sL_U(2)$ given by $(f, h) \mapsto f \circ
(h \oplus \id)$ and equipped with the right action of $\sL_U(1)$
specified by $([f],h) \mapsto [f \circ (\id \oplus h)]$, $\theta_2$ is
the restriction to the second summand, and $\theta_1$ is specified by
$(g,0,f) \mapsto g \circ (\id \oplus f)$.  Both maps are
$G \times \sL_U(1)$ maps, and $\theta_1$ is a homeomorphism.

Now take $k > 1$.  Then $\gamma_k$ factors as the composite
\[
\hat{\sL}_U(k) \cong (\sL_U(2)
/ \sL_U(1)) \times_{\sL_U(1)} \sL_U(k) \to
\sL_U(1) \times_{\sL_U(1)} \sL_U(k) \cong \sL_U(k),
\]
induced by $\gamma_1$, where we are using the homeomorphism
\[
\hat{\sL}_U(k) \cong (\sL_U(2) \times_{\sL_U(1) \times \sL_U(1)}
(\sL_U(0) \times \sL_U(1))) \times_{\sL_U(1)} \sL_U(k).
\]
To see this, observe that $\gamma_k((g,0,f)) = g \circ (f \oplus 0)$.
On the other hand, the composite above first takes $(g,0,f)$ to
$((g \circ \id), f)$, then $((g \circ \id),f)$ to $(\theta_2(g), f)$,
and finally $(\theta_2(g),f)$ to $\theta_2(g) \circ f = g \circ
(f \oplus 0)$.

Since $\sL_U(k)$ is a universal space for the family of subgroups of
$G \times \Sigma_k$ prescribed by $U$, it suffices to show that
$(\sL_U(2)/\sL_U(1)) \times_{\sL_U(1)} \sL_U(k)$ is also a universal
space for the same family.  To do this, we will unpack part of the
proof of~\cite[XI.2.2]{EKMM}.

Write $U \cong U_1 \oplus U_2$ as $G$-spaces, where $U_1$ and $U_2$
are $G$-universes such that $U_1 \cong U$ and $U_2 \cong U$; we can do
this by Lemma~\ref{lem:unisum}.  Define
$\sK(2) \subset \sL_U(2)$ to be $\{f \, \mid \, f(\{0\}\oplus
U) \subset U_2$, equipped with the conjugation $G$-action.  Next, we
define
\[
\hat{\sK}_1 = \sK(2) / \sL_U(1)
\]
and we let $\sK_1 \subset \sL_U(1)$ be $\{f \, \mid \, f(U) \subseteq
U_2\}$ with the conjugation $G$-action.  The map $\theta_2$ restricts
to give a $G$-map $\hat{\sK}_1 \to \sK_1$ which is compatible with the
action of $\sL_U(1)$ and so by Lemma~\ref{lem:action} we have an induced
$G \times \Sigma_k$-map
\[
\hat{\sK}_1 \times_{\sL_U(1)} \sL_U(k) \to \sK_1 \times_{\sL_U(1)} \sL_U(k).
\]
The non-equivariant argument in~\cite[XI.2.2]{EKMM} extends to the
equivariant case to show that the map $\hat{\sK}_1 \to \sK_1$ is a
homeomorphism of $G$-spaces.  On the other hand, we have a
homeomorphism of $G$-spaces $\sK_1 \cong \sL(U,U_2) \cong \sL_U(1)$
which is compatible with the action of $\sL_U(1)$, and so
Lemma~\ref{lem:action} implies that there is a composite
$G \times \Sigma_k$-map
\[
\sK_1 \times_{\sL_U(1)} \sL_U(k) \to \sL_U(1) \times_{\sL_U(1)} \sL_U(k) \cong \sL_U(k)
\]
which is a homeomorphism.  Putting these together, we have a
$G \times \Sigma_k$-map
\[
\hat{\sK}_1 \times_{\sL_U(1)} \sL_U(k) \cong \sL_U(k)
\]
that is a homeomorphism.

To finish the argument, observe that the proof in~\cite[XI.2.2]{EKMM}
extends to the equivariant context to show that the inclusion
\[
\sK(2) \to \sL_U(2)
\]
is a $G$-homotopy equivalence of right
$\sL_U(1) \times \sL_U(1)$-spaces and therefore
\[
\hat{\sK}_1 \to (\sL_U(2)/\sL_U(1))
\]
is a $G$-homotopy equivalence of right $\sL_U(1)$-spaces.  As a
consequence, the induced map
\[
\hat{\sK}_1 \times_{\sL_U(1)} \sL_U(k) \to
(\sL_U(2)/\sL_U(1)) \times_{\sL_U(1)} \sL_U(k)
\]
is a $G \times \Sigma_k$-homotopy equivalence.
\end{proof}

\section{Compact Lie groups}

In this appendix, we quickly outline what aspects of our work in this
paper continue to hold when $G$ is an infinite compact Lie group.
Basically, all of the foundational material in this paper goes through
except the results on multiplicative norms; when $G$ is an infinite
compact Lie group, norms exist only for subgroups $H$ of finite index
and hence we can only work with admissible finite sets.  With this
modification, the theorems of the paper remain true.

To be more precise, the work of the paper depends on various results
about the linear isometries operad, mostly collected in
Appendix~\ref{app:lin}.  Lemma~\ref{lem:unisum} holds with the same
proof for finite $G$-sets; however, in all of our
applications of Lemma~\ref{lem:unisum}, this case suffices.
Lemmas~\ref{lem:equikriz} and~\ref{lem:normkriz} hold with the same
proofs; these arguments do not rely on the finiteness of $G$.
Lemmas~\ref{lem:normhop} and~\ref{lem:lemhop2} again require finite
$G$-sets, but this suffices to conclude Lemmas~\ref{cor:corhop}
and~\ref{cor:corhop2}, respectively.  In the body of the paper,
Theorem~\ref{thm:multichange} goes through with the same proof, as
does the essential Theorem~\ref{thm:mandell}.

As a consequence, the work of the remainder of the paper goes through
without modification in the arguments except for the material on the
norm in Sections~\ref{sec:pointnorm},~\ref{sec:homnorm},
and~\ref{sec:Gsym}.  Here, the results on $N_H^G$ require that $G/H$
be a finite $G$-set, i.e., that the subgroups have finite index.

\end{appendix}

\end{document}